\documentclass[a4paper,11pt]{article}
\usepackage{amsmath,amsthm}
\usepackage{authblk}
\usepackage[style=numeric-comp,maxbibnames=99,bibencoding=utf8,firstinits=true]{biblatex}
\usepackage{booktabs}
\usepackage{cancel}
\usepackage{cases}
\usepackage{comment}
\usepackage[hmargin={26mm,26mm},vmargin={30mm,35mm}]{geometry}
\usepackage{hyperref}
\usepackage{paralist}
\usepackage[compact]{titlesec}
\usepackage{xcolor}

\usepackage{newtxtext}
\usepackage{newtxmath}
\usepackage{graphicx}
\usepackage{subcaption}
 
\bibliography{rmddr}

\usepackage{pgfplots,pgfplotstable}
\usepackage{tikz-cd}

\graphicspath{{figures/}}

\newcommand{\logLogSlopeTriangle}[5]
{
    \pgfplotsextra
    {
        \pgfkeysgetvalue{/pgfplots/xmin}{\xmin}
        \pgfkeysgetvalue{/pgfplots/xmax}{\xmax}
        \pgfkeysgetvalue{/pgfplots/ymin}{\ymin}
        \pgfkeysgetvalue{/pgfplots/ymax}{\ymax}

        \pgfmathsetmacro{\xArel}{#1}
        \pgfmathsetmacro{\yArel}{#3}
        \pgfmathsetmacro{\xBrel}{#1-#2}
        \pgfmathsetmacro{\yBrel}{\yArel}
        \pgfmathsetmacro{\xCrel}{\xArel}

        \pgfmathsetmacro{\lnxB}{\xmin*(1-(#1-#2))+\xmax*(#1-#2)} 
        \pgfmathsetmacro{\lnxA}{\xmin*(1-#1)+\xmax*#1} 
        \pgfmathsetmacro{\lnyA}{\ymin*(1-#3)+\ymax*#3} 
        \pgfmathsetmacro{\lnyC}{\lnyA+#4*(\lnxA-\lnxB)}
        \pgfmathsetmacro{\yCrel}{\lnyC-\ymin)/(\ymax-\ymin)}

        \coordinate (A) at (rel axis cs:\xArel,\yArel);
        \coordinate (B) at (rel axis cs:\xBrel,\yBrel);
        \coordinate (C) at (rel axis cs:\xCrel,\yCrel);

        \draw[#5]   (A)-- node[pos=0.5,anchor=north] {\scriptsize{1}}
                    (B)-- 
                    (C)-- node[pos=0.,anchor=west] {\scriptsize{#4}} 
                    cycle;
    }
}

\newcommand{\email}[1]{\href{mailto:#1}{#1}}

\newtheorem{theorem}{Theorem}

\newtheorem{lemma}[theorem]{Lemma}

\theoremstyle{remark}
\newtheorem{remark}[theorem]{Remark}
\theoremstyle{definition}

\newcommand{\st}{\,:\,}
\newcommand{\Real}{\mathbb{R}}

\DeclareRobustCommand{\bvec}[1]{\boldsymbol{#1}}
\newcommand{\uvec}[1]{\underline{\bvec{#1}}}
\newcommand{\tvec}[1]{\widetilde{\bvec{#1}}}
\newcommand{\cvec}[1]{\bvec{\mathcal{#1}}}
\newcommand{\tens}[1]{\boldsymbol{\mathsf{#1}}}

\DeclareMathOperator{\GRAD}{\bf grad}
\DeclareMathOperator{\GRADs}{\boldsymbol{\mathsf{grad}}_{\rm s}}
\DeclareMathOperator{\GRADss}{\boldsymbol{\mathsf{grad}}_{\rm ss}}
\DeclareMathOperator{\CURL}{\bf curl}
\DeclareMathOperator{\DIV}{div}
\DeclareMathOperator{\vDIV}{\bf div}
\DeclareMathOperator{\ROT}{rot}
\DeclareMathOperator{\VROT}{\bf rot}
\DeclareMathOperator{\tr}{tr}

\newcommand{\HDrot}[1]{\bvec{H}_0(\ROT;#1)}

\newcommand{\compl}{{\rm c}}

\newcommand{\edges}[1]{\mathcal{E}_{#1}}
\newcommand{\vertices}[1]{\mathcal{V}_{#1}}

\newcommand{\ET}{\edges{T}}

\newcommand{\VT}{\vertices{T}}

\newcommand{\normal}{\bvec{n}}
\newcommand{\tangent}{\bvec{t}}

\newcommand{\Poly}[2][]{\mathcal{P}_{#1}^{#2}}
\newcommand{\vPoly}[2][]{\cvec{P}_{#1}^{#2}}
\newcommand{\tPoly}[2][]{\tens{P}_{#1}^{#2}}
\newcommand{\Roly}[1]{\cvec{R}^{#1}}
\newcommand{\cRoly}[1]{\cvec{R}^{\compl,#1}}

\newcommand{\lproj}[2]{\pi_{\mathcal{P},#2}^{#1}}
\newcommand{\vlproj}[2]{\boldsymbol{\pi}_{\cvec{P},#2}^{#1}}
\newcommand{\tlproj}[2]{\boldsymbol{\pi}_{\tens{P},#2}^{#1}}

\newcommand{\strain}{\tens{\varepsilon}}
\newcommand{\sproj}[2][T]{\bvec{\pi}_{\strain,#1}^{#2}}

\newcommand{\Rproj}[2]{\bvec{\pi}_{\cvec{R},#2}^{#1}}
\newcommand{\cRproj}[2]{\bvec{\pi}_{\cvec{R},#2}^{\compl,#1}}

\newcommand{\term}{\mathfrak{T}}

\newcommand{\Mh}{\mathcal{M}_h}
\newcommand{\Th}{\mathcal{T}_h}

\newcommand{\Ehi}{\Eh^{{\rm i}}}
\newcommand{\Ehb}{\Eh^{{\rm b}}}
\newcommand{\Eh}{\mathcal{E}_h}
\newcommand{\Vh}{\mathcal{V}_h}
\newcommand{\ST}{\mathcal{S}_T}

\newcommand{\Irot}[1]{\uvec{I}_{\bvec{\Theta},#1}^k}
\newcommand{\Irotz}[1]{\uvec{I}_{\bvec{\Theta},#1}^0}
\newcommand{\tIrot}[1]{\uvec{I}_{\bvec{\Theta},#1}^{\flat,k}}
\newcommand{\tIrotz}[1]{\uvec{I}_{\bvec{\Theta},#1}^{\flat,0}}
\newcommand{\Idis}[1]{\underline{I}_{U,#1}^k}
\newcommand{\Idisz}[1]{\underline{I}_{U,#1}^0}

\newcommand{\tGT}{\tens{G}_T^k}
\newcommand{\GsT}{\tens{G}_{{\rm s},T}^k}
\newcommand{\DT}{D_T^k}
\newcommand{\RT}{R_T^k}
\newcommand{\PT}[1][k]{\bvec{P}_{\bvec{\Theta},T}^{#1}}
\newcommand{\injHHO}[1]{\underline{\boldsymbol{\mathfrak{I}}}_{{\rm HHO},#1}^k}
\newcommand{\pTHHO}[1][k+1]{\bvec{p}_{T}^{#1}}
\newcommand{\pToneHHO}[1][k+1]{\bvec{p}_{T_1}^{#1}}
\newcommand{\pTtwoHHO}[1][k+1]{\bvec{p}_{T_2}^{#1}}
\newcommand{\phHHO}[1][k+1]{\bvec{p}_h^{#1}}

\newcommand{\Gsh}[1][k]{\tens{G}_{{\rm s},h}^{#1}}
\newcommand{\Dh}{D_h^k}

\newcommand{\GT}[1][k]{\bvec{G}_T^{#1}}

\newcommand{\uGh}[1][k]{\uvec{G}_h^{#1}}
\newcommand{\uGT}[1][k]{\uvec{G}_T^{#1}}
\newcommand{\PUT}[1][k+1]{P_{U,T}^{#1}}
\newcommand{\PUh}[1][k+1]{P_{U,h}^{#1}}

\newcommand{\norm}[2][]{\|#2\|_{#1}}
\newcommand{\seminorm}[2][]{|#2|_{#1}}
\newcommand{\vvvert}{\vert\kern-0.25ex\vert\kern-0.25ex\vert}

\DeclareMathOperator{\card}{card}

\begin{document}

\title{A discrete de Rham method for the Reissner--Mindlin plate bending problem on polygonal meshes}
\author[1]{Daniele A. Di Pietro}
\author[2]{J\'er\^ome Droniou}
\affil[1]{IMAG, Univ Montpellier, CNRS, Montpellier, France, \email{daniele.di-pietro@umontpellier.fr}}
\affil[2]{School of Mathematics, Monash University, Melbourne, Australia, \email{jerome.droniou@monash.edu}}

\maketitle

\begin{abstract}
  In this work we propose a discretisation method for the Reissner--Mindlin plate bending problem in primitive variables that supports general polygonal meshes and arbitrary order.
  The method is inspired by a two-dimensional discrete de Rham complex for which key commutation properties hold that enable the cancellation of the contribution to the error linked to the enforcement of the Kirchhoff constraint.
  Denoting by $k\ge 0$ the polynomial degree for the discrete spaces and by $h$ the meshsize, we derive for the proposed method an error estimate in $h^{k+1}$ for general $k$, as well as a locking-free error estimate for the lowest-order case $k=0$.
  The theoretical results are validated on a complete panel of numerical tests.
\end{abstract}

\section{Introduction}

In this work we propose a novel discretisation method for the Reissner--Mindlin plate bending problem in primitive variables that supports general polygonal meshes and arbitrary order.
In its lowest-order version, the method can be proved to behave robustly with respect to the plate thickness $t$.
Its design is based on the two-dimensional discrete de Rham (DDR) complex of \cite{Di-Pietro.Droniou:21}, for which key commutation properties hold that enable the cancellation of the contribution to the error linked to the enforcement of the Kirchhoff constraint.

We consider in what follows an elastic plate of thickness $t>0$ with reference configuration $\Omega\times\left(-\frac{t}2,\frac{t}2\right)$, where $\Omega\subset\Real^2$ is a bounded connected polygonal domain with boundary $\partial\Omega$.
Without loss of generality, it is assumed in what follows that $\Omega$ has diameter 1 and that $t<1$.
The Reissner--Mindlin model describes the deformation of the plate in terms of the rotation $\bvec{\theta}:\Omega\to\Real^2$ of the fibers initially perpendicular to its midsurface and of the transverse displacement $u:\Omega\to\Real$.
Introducing the shear strain $\tens{\gamma}$ and denoting by $f:\Omega\to\Real$ the transverse load, the strong formulation of the model with clamped boundary conditions reads
\begin{subequations}\label{eq:rm.strong}
  \begin{alignat}{4}
    \label{eq:rm.balance1}
    \tens{\gamma}+\vDIV(\tens{C}\GRADs\bvec{\theta}) &= \bvec{0}&\qquad&\text{in $\Omega$},\\
    \label{eq:rm.balance2}
    -\DIV\bvec{\gamma} &=f &\qquad&\text{in $\Omega$},\\
    \label{eq:rm.def.gamma}
    \tens{\gamma} &= \frac{\kappa}{t^2}(\GRAD u-\bvec{\theta}) &\qquad&\text{in $\Omega$},\\
    \label{eq:rm.bc}
    \bvec{\theta} &= \bvec{0},\quad u = 0 &\qquad&\text{on $\partial\Omega$}.
  \end{alignat}
\end{subequations}
Here, $\vDIV$ is the row-wise divergence of tensors, $\GRADs$ is the symmetric part of the gradient applied to vector-valued fields over $\Omega$, and $\tens{C}$ is the fourth-order tensor defined by $\tens{C}\tens{t} = \beta_0\tens{t} + \beta_1(\tr\tens{t})\tens{I}$ for all second-order tensor $\tens{t}$, with $\tens{I}$ the identity tensor. The parameters of $\tens{C}$ are $\beta_0\coloneq\frac{E}{12(1+\nu)}$ and $\beta_1\coloneq\frac{E\nu}{12(1-\nu^2)}$, where $E>0$ and $\nu\in\lbrack0,\frac12)$ are the Young modulus and Poisson ratio of the material, respectively.
The shear modulus $\kappa$ is given as $\kappa\coloneq\frac{\kappa_0E}{2(1+\nu)}$, with shear correction factor $\kappa_0$ usually taken equal to $\frac56$ for clamped plates (for a discussion on the choice of this correction factor in various configurations see, e.g., \cite{Lim:16} and references therein).
Notice that both the method and its analysis could be extended to non-isotropic materials and spatially varying coefficients (using the same approach as in \cite{Di-Pietro.Droniou:20*1} where a DDR scheme for a magnetostatic model with varying permeability is considered), but we avoid delving further into this topic here in order to keep the exposition as simple as possible.%
Denoting by $H_0^1(\Omega)$ the space of real-valued functions that are square-integrable along with their derivatives and that vanish on $\partial\Omega$ in the sense of traces, the standard weak formulation of \eqref{eq:rm.strong} hinges on the spaces $\bvec{\Theta}\coloneq H_0^1(\Omega)^2$ for the rotation and $U\coloneq H_0^1(\Omega)$ for the transverse displacement.
Specifically, assuming that $f\in L^2(\Omega)$, it reads:
Find $(\bvec{\theta},u)\in\bvec{\Theta}\times U$ such that
\begin{equation}\label{eq:weak}
  A((\bvec{\theta},u),(\bvec{\eta},v)) = \ell(v)\qquad\forall(\bvec{\eta},v)\in\bvec{\Theta}\times U,
\end{equation}
where the bilinear form $A:\left[\bvec{\Theta}\times U\right]^2\to\Real$ and the linear form $\ell:U\to\Real$ are such that, for all $(\bvec{\tau},w),(\bvec{\eta},v)\in\bvec{\Theta}\times U$,
\[
A((\bvec{\tau},w),(\bvec{\eta},v)) \coloneq a(\bvec{\tau},\bvec{\eta}) + b((\bvec{\tau},w),(\bvec{\eta},v)),\qquad
\ell(v)\coloneq\int_\Omega f v,
\]
with bilinear forms $a:\bvec{\Theta}\times\bvec{\Theta}\to\Real$ and $b:\left[\bvec{\Theta}\times U\right]^2\to\Real$ such that
\begin{align}
  \label{eq:bilinear.a}
  a(\bvec{\tau},\bvec{\eta})
  &\coloneq \beta_0\int_\Omega\GRADs\bvec{\tau}:\GRADs\bvec{\eta}
  + \beta_1\int_\Omega\DIV\bvec{\tau}~\DIV\bvec{\eta},\qquad
  \\
  \label{eq:bilinear.b}
  b((\bvec{\tau},w),(\bvec{\eta},v))
  &\coloneq
  \frac{\kappa}{t^2}\int_\Omega(\bvec{\tau} - \GRAD w)\cdot(\bvec{\eta} - \GRAD v).
\end{align}
The role of the bilinear form $b$ is to enforce the Kirchhoff constraint that, as $t\to 0$, the rotation of the normal fibers equals the gradient of the transverse displacement.
Notice that the choice of considering clamped boundary conditions is made for the sole purpose of simplifying the theoretical discussion: other standard boundary conditions can be considered with straightforward modifications.
A critical point in the numerical approximation of problem \eqref{eq:weak} is robustness for small $t$.
Methods for which error estimates uniform in $t$ can be established are commonly referred to as (shear) \emph{locking-free}.

The finite element literature for the locking-free discretisation of problem \eqref{eq:weak} on standard meshes dates back to the 1980s.
In \cite{Brezzi.Fortin:86}, the authors proposed a reformulation involving, in addition to the primitive variables $\bvec{\theta}$ and $u$, the introduction of two additional variables corresponding to the irrotational and solenoidal parts of the transverse shear strain.
This work pointed out the relevance of establishing a discrete version of the Helmholtz decomposition to obtain error estimates uniform in $t$.
A method in primitive variables was later proposed in \cite{Arnold.Falk:89}, based on a nonconforming (Crouzeix--Raviart) piecewise linear space for the displacement and a bubble-enriched continuous space for the rotation, and involving a projection in the discrete version of the bilinear form $b$.
Recent developments of these ideas, including the extension to higher orders and the use of the Taylor--Hood element pair for the underlying Stokes problem, can be found in \cite{Schedensack:17,Gallistl.Schedensack:21}.
The idea of  using reduced integration or projections in the enforcement of the Kirchhoff constraint can be found in several other works; see, e.g., \cite{Brezzi.Fortin.ea:91,Duran.Liberman:92,Onate.Zarate.ea:94,Arnold.Falk:97,Lamichhane.Meylan:17} and also \cite{Pitkaranta.Suri:00}, where upper and lower bounds for a variety of methods including the stabilised scheme of \cite{Chapelle.Stenberg:98} are derived.
A different approach, resorting to a mixed formulation where the shear strain appears as a separate unknown, is considered in \cite{Arnold.Brezzi:93}.
The key point is, in this case, the design of a suitable coupling bilinear form, for which abstract conditions are provided.
Recent results on mixed finite element schemes can be found in \cite{Lamichhane:13}; see also the references therein.
Mixed approaches inspired by fully nonconforming (discontinuous Galerkin) methods have been proposed in \cite{Arnold.Brezzi.ea:05}, later leading to choices of finite element spaces that do not require reduced integration \cite{Arnold.Brezzi.ea:07}; see also \cite{Lovadina:05,Chinosi.Lovadina.ea:06} for related developments.
  Discontinuous Galerkin methods in their weakly over-penalised symmetric formulation are considered in \cite{Bosing.Carstensen:15,Bosing.Carstensen:15*1}.

While the use of standard (e.g., simplicial conforming) meshes can be satisfactory for simple geometries and problems, it may lack flexibility in more complex situations.
The support of general meshes can greatly simplify the meshing process in the presence of small geometric features \cite{Antonietti.Giani.ea:13} and pave the way for advanced techniques such as nonconforming adaptive mesh refinement (which does not trade mesh quality for size) and mesh coarsening \cite{Bassi.Botti.ea:12,Bassi.Botti.ea:14,Di-Pietro.Specogna:16}, that are crucial to exploit high-order approximations in the presence of geometric singularities.
Owing to the onset of polygonal elements and/or hanging nodes, such strategies are inaccessible to standard conforming finite elements (although, in specific cases, variations to support them can be conceived \cite{Anand.Ovall.ea:20}).
These and similar considerations have prompted, in the last few years, the development of locking-free discretisation methods for problem \eqref{eq:weak} supporting general polygonal meshes.
A first example is provided by the low-order Mimetic Finite Difference method of \cite{Beirao-da-Veiga.Mora:11}, that hinges on transverse displacements defined at mesh vertices, rotations defined at mesh vertices and edges, and uses shear forces at edges as intermediate unknowns.
The key ingredient to establish a first-order locking-free error estimate is once again a discrete Helmholtz decomposition.
A lowest-order Virtual Element method has also been recently proposed in \cite{Beirao-da-Veiga.Mora.ea:19}, inspired by the reformulation of problem \eqref{eq:weak} originally introduced in \cite{Beirao-da-Veiga.Hughes.ea:15} in the context of Isogeometric Analysis and using the transverse displacement and shear strain as unknowns.

The DDR method proposed in this work contains several key elements of novelty.
First, to the best of our knowledge, it is the first scheme to support general polygonal meshes and high-order.
Second, it does not resort to reduced integration or projections in the discrete counterpart of the bilinear form $b$.
Third, it admits an inexpensive lowest-order version for which locking-free estimates can be rigorously established.
The starting point for the design of the scheme is the two-dimensional DDR complex of \cite[Remark 13]{Di-Pietro.Droniou:21}.
This complex satisfies a crucial commutation property between the reconstructions of the discrete displacement gradient, the continuous gradient, and the interpolators on the corresponding spaces; see \eqref{eq:tIrot:P2} below.
When performing a convergence analysis in the spirit of the Third Strang Lemma \cite{Di-Pietro.Droniou:18}, one can leverage this commutation property to cancel the error resulting from the enforcement of the Kirchhoff constraint through the discrete counterpart of the bilinear form $b$.
This remark suggests the use of DDR counterparts of the $H^1_0(\Omega)$ and $\HDrot{\Omega}$ spaces for the displacement and the rotation, respectively.
In order to have sufficient information to reconstruct a full strain tensor, the discrete $\HDrot{\Omega}$ space has to be enriched by the addition of normal components at edges.
It turns out that this enriched space can be embedded into the standard Hybrid High-Order (HHO) space for elasticity originally introduced in \cite{Di-Pietro.Ern:15} (see also \cite[Chapter 7]{Di-Pietro.Droniou:22}), so that the standard HHO construction can be exploited to design the discrete counterpart of the bilinear form $a$.
We notice, in passing, that HHO and DDR schemes for Kirchhoff--Love plates can be found, respectively, in \cite{Bonaldi.Di-Pietro.ea:18} and \cite{Di-Pietro.Droniou:22}.
With these ingredients, we establish in Theorem \ref{th:error.est} an estimate in $h^{k+1}$ (with $h$ denoting the meshsize and $k$ the polynomial degree of the DDR sequence) for the natural (coercivity) norm of the error.
The right-hand side of this estimate does not explicitly depend on $t$, but involves, as is unavoidable for high-order schemes, norms of higher order derivatives of the strain; such norms are not expected to remain bounded as $t\to 0$.
Through the introduction of novel liftings of the displacement and of the rotation, we show in Theorem \ref{th:error.est.k0} that an error estimate uniform in $t$ (and, thus, locking-free) can be established in the lowest order case $k=0$, under smoothness and boundedness assumptions on the solution that are realistic in applications.
  It is worth pointing out that, in the light of the results recently obtained in \cite[Section 6]{Beirao-da-Veiga.Dassi.ea:22}, the present method is expected to admit a Virtual Element reformulation and, therefore, be interpreted as a (generalised) finite element method based on computable projections of the basis functions and their derivatives.

The rest of the paper is organised as follows.
In Section \ref{sec:setting} we introduce the discrete setting.
Section \ref{sec:scheme} contains the statement of the discrete problem preceded by the required constructions.
The analysis of the method is carried out in Section \ref{sec:analysis}, the main theorems being stated in Section \ref{sec:results} and their proofs given in Sections \ref{sec:proof.th.error.est} and \ref{sec:proof.th.error.est.k0}.
Finally, Section \ref{sec:numerical.results} contains a complete panel of numerical results, introducing a novel analytical solution for the model, showing that the method displays, to a certain extent, a locking-free behaviour also for $k\ge 1$, and comparing in the Kirchoff--Love limit $t\to 0$ the DDR scheme to a stabilised finite element scheme.


\section{Setting}\label{sec:setting}

\subsection{Mesh}

For any measurable set $Y\subset\Real^2$, we denote by $h_Y\coloneq\sup\{|\bvec{x}-\bvec{y}|\st \bvec{x},\bvec{y}\in Y\}$ its diameter and by $|Y|$ its Hausdorff measure.
We consider meshes $\Mh\coloneq\Th\cup\Eh\cup\Vh$, where:
$\Th$ is a finite collection of open disjoint polygonal elements such that $\overline{\Omega} = \bigcup_{T\in\Th}\overline{T}$ and $h=\max_{T\in\Th}h_T>0$;
$\Eh$ is the set collecting the open polygonal edges (line segments) of the elements;
$\Vh$ is the set collecting the edge endpoints.
It is assumed, in what follows, that $(\Th,\Eh)$ matches the conditions in \cite[Assumption 7.6]{Di-Pietro.Droniou:20}.
The sets collecting the mesh edges that lie on the boundary of a mesh element $T\in\Th$ and on $\partial\Omega$ are denoted by $\ET$ and $\Ehb$, respectively. We also denote by $\Ehi=\Eh\setminus\Ehb$ the set of internal edges.
The coordinates vector of $V\in\Vh$ is denoted by $\bvec{x}_V$.

Each $E\in\Eh$ is endowed with an orientation determined by a fixed unit tangent vector $\tangent_E$; we then choose the unit normal $\normal_E$ such that $(\tangent_E,\normal_E)$ forms a right-hand system of coordinates. For $T\in\Th$ and $E\in\ET$, we set $\omega_{TE}=1$ if $\tangent_E$ points in the clockwise direction of $\partial T$, and $\omega_{TE}=-1$ otherwise.
It can be checked that $\normal_{TE}\coloneq\omega_{TE}\normal_E$ is the outer unit normal to $T$ on $E$.

\subsection{Polynomial spaces}

For any $Y\in\Th\cup\Eh$, we denote by $\Poly{\ell}(Y)$ the space spanned by the restriction to $Y$ of two-variate polynomials of total degree $\le\ell$,  with the convention that $\Poly{-1}(Y)=\{0\}$.
We additionally denote by $\lproj{\ell}{Y}$ the corresponding $L^2$-orthogonal projector.
For all $E\in\Eh$, the space $\Poly{\ell}(E)$ is isomorphic to univariate polynomials of total degree $\le\ell$ (see \cite[Proposition 1.23]{Di-Pietro.Droniou:20}).
In what follows, with a little abuse of notation, both spaces are denoted by $\Poly{\ell}(E)$.
For $Y\in\Th\cup\Eh$, the vector and tensor versions of $\Poly{\ell}(Y)$ are respectively denoted by $\vPoly{\ell}(Y)\coloneq\Poly{\ell}(Y)^2$ and $\tPoly{\ell}(Y)\coloneq\Poly{\ell}(Y)^{2\times 2}$, and the corresponding $L^2$-orthogonal projectors $\vlproj{\ell}{Y}$ and $\tlproj{\ell}{Y}$ are obtained applying $\lproj{\ell}{Y}$ component-wise.
We additionally denote by $\tPoly[\rm s]{\ell}(Y)$ the subspace of symmetric-valued functions in $\tPoly{\ell}(Y)$.

For all $T\in\Th$, let $\bvec{x}_T\in T$ be such that $T$ contains a ball centered at $\bvec{x}_T$ of radius $\rho h_T$, where $\rho$ is the mesh regularity parameter in \cite[Assumption 7.6]{Di-Pietro.Droniou:20}.
For any integer $\ell\ge 0$, we define the following relevant subspaces of $\vPoly{\ell}(T)$:
\begin{equation}\label{eq:spaces.T}
    \Roly{\ell}(T)\coloneq\VROT\Poly{\ell+1}(T),
    \qquad
    \cRoly{\ell}(T)\coloneq(\bvec{x}-\bvec{x}_T)\Poly{\ell-1}(T).
\end{equation}
We have
\begin{equation}\label{eq:decomposition:Poly.ell.F}
  \vPoly{\ell}(T) = \Roly{\ell}(T) \oplus \cRoly{\ell}(T).
\end{equation}
Notice that the direct sums in the above expression are not $L^2$-orthogonal in general.
The $L^2$-orthogonal projectors on the spaces \eqref{eq:spaces.T} are, with obvious notation, $\Rproj{\ell}{T}$, and $\cRproj{\ell}{T}$.


\section{DDR scheme}\label{sec:scheme}

The scheme for \eqref{eq:weak} is designed using spaces of unknowns from the DDR method \cite{Di-Pietro.Droniou:21} together with an enrichment inspired by HHO methods \cite{Di-Pietro.Droniou:20}.

\subsection{Spaces and interpolators}

Let a polynomial degree $k\ge 0$ be fixed and set
\[
\begin{aligned}
  \uvec{\Theta}_h^k
  &\coloneq\Big\{
  \uvec{\eta}_h=\big(
  (\bvec{\eta}_{\cvec{R},T},\bvec{\eta}_{\cvec{R},T}^\compl)_{T\in\Th},
  (\bvec{\eta}_E)_{E\in\Eh}
  \big)\st
  \begin{aligned}[t]
    &\text{$(\bvec{\eta}_{\cvec{R},T},\bvec{\eta}_{\cvec{R},T}^\compl)\in\Roly{k-1}(T)\times\cRoly{k}(T)$ for all $T\in\Th$,}
    \\
    &\text{and $\bvec{\eta}_E\in\vPoly{k}(E)$ for all $E\in\Eh$}\Big\},
  \end{aligned}
  \\
  \underline{U}_h^k
  &\coloneq\Big\{
  \underline{v}_h=\big(
  (v_T)_{T\in\Th}, v_{\Eh}
  \big)\st
  \text{$v_T\in\Poly{k-1}(T)$ for all $T\in\Th$ and $v_{\Eh}\in\Poly[\rm c]{k+1}(\Eh)$}
  \Big\},
\end{aligned}
\]
where $\Poly[\rm c]{k+1}(\Eh)$ is spanned by the functions over the mesh edge skeleton whose restriction to each edge $E\in\Eh$ is a polynomial of total degree $\le k+1$ and that are continuous at the edges endpoints.
The space $\uvec{\Theta}_h^k$ is an enrichment of the two-dimensional DDR space $\uvec{X}^k_{\CURL,h}$ with edge unknowns representing a full vector-valued field as opposed to its tangent component only; the space $\underline{U}_h^k$ coincides with the two-dimensional DDR space $\uvec{X}^k_{\GRAD,h}$.

Smooth functions are interpolated as follows:
For all $\bvec{\eta}\in H^1(\Omega)^2$ 
\begin{equation}\label{eq:Irot.h}
  \Irot{h}\bvec{\eta}\coloneq\big(
  (\Rproj{k-1}{T}\bvec{\eta}_{|T},\cRproj{k}{T}\bvec{\eta}_{|T})_{T\in\Th},
  (\vlproj{k}{E}\bvec{\eta}_{|E})_{E\in\Eh}
  \big)\in\uvec{\Theta}_h^k,
\end{equation}
while, for all $v\in C^0(\overline{\Omega})$,
\[
  \begin{gathered}
    \Idis{h} v \coloneq\big(
    (\lproj{k-1}{T} v_{|T})_{T\in\Th}, v_{\Eh}
    \big)\in \underline{U}_h^k,
    \\
    \text{%
      with $\lproj{k-1}{E}(v_{\Eh})_{|E} = \lproj{k-1}{E} v_{|E}$ for all $E\in\Eh$
      and $v_{\Eh}(\bvec{x}_V) = v(\bvec{x}_V)$ for all $V\in\Vh$.
    }
  \end{gathered}
\]

For all $T\in\Th$, we denote by $\uvec{\Theta}_T^k$ and $\underline{U}_T^k$, respectively, the restrictions of $\uvec{\Theta}_h^k$ and $\underline{U}_h^k$ to $T$, collecting the polynomial components that lie inside $T$ and on its boundary.
A similar convention is adopted for the elements of these spaces and for the interpolators.

\subsection{Discrete differential operators and potentials}

We introduce discrete versions of the differential operators and of the rotation field reconstructed from the unknowns in the discrete spaces.

\subsubsection{Discrete gradient and transverse displacement reconstruction on $\underline{U}_T^k$}

We follow here standard constructions from the DDR method.
For all $T\in\Th$, the polynomial transverse displacement gradient $\GT:\underline{U}_T^k\to\vPoly{k}(T)$ is such that, for all $\underline{v}_T\in\underline{U}_T^k$,
\begin{equation}\label{eq:def.GT}
\int_T\GT\underline{v}_T\cdot\bvec{\eta}
= -\int_Tv_T\DIV\bvec{\eta}
+ \sum_{E\in\ET}\omega_{TE}\int_E v_{\ET}(\bvec{\eta}\cdot\normal_E)
\qquad\forall\bvec{\eta}\in\vPoly{k}(T).
\end{equation}
We additionally define the transverse displacement reconstruction $\PUT:\underline{U}_T^k\to\Poly{k+1}(T)$ such that, for all $\underline{v}_T\in\underline{U}_T^k$,
\[
\int_T\PUT\underline{v}_T\DIV\bvec{\eta}
= -\int_T\GT\underline{v}_T\cdot\bvec{\eta}
+ \sum_{E\in\ET}\omega_{TE}\int_Ev_{\ET}(\bvec{\eta}\cdot\normal_E)
\qquad\forall\bvec{\eta}\in\cRoly{k+2}(T).
\]
A global transverse displacement reconstruction is obtained setting, for all $\underline{v}_h\in\underline{U}_h^k$,
\[
(\PUh\underline{v}_h)_{|T}\coloneq\PUT\underline{v}_T\qquad\forall T\in\Th.
\]
Finally, we define a global discrete transverse displacement gradient $\uGh:\underline{U}_h^k\to\uvec{\Theta}_h^k$ as follows:
For all $\underline{v}_h\in\underline{U}_h^k$,
\[
\uGh\underline{v}_h\coloneq\big(
(\Rproj{k-1}{T}\GT\underline{v}_T, \cRproj{k}{T}\GT\underline{v}_T)_{T\in\Th},
((v_{\Eh})_{|E}'\tangent_E)_{E\in\Eh}
\big),
\]
where the derivative along the edge is taken in the direction of $\tangent_E$.

To state the key commutation property used to prove the error estimates for the DDR scheme, we need to introduce a modified version of the interpolator on $\uvec{\Theta}_h^k$, which is adjusted to the account for the fact that, on the edges, the discrete gradient only encodes the tangential derivatives. The modified interpolator is $\tIrot{h}:H^1(\Omega)^2\to \uvec{\Theta}_h^k$ such that, for all $\bvec{\eta}\in H^1(\Omega)^2$,
\begin{equation}\label{eq:def.tIrot}
  \tIrot{h}\bvec{\eta}\coloneq\big(
  (\Rproj{k-1}{T}\bvec{\eta}_{|T},\cRproj{k}{T}\bvec{\eta}_{|T})_{T\in\Th},
  (\lproj{k}{E}(\bvec{\eta}_{|E}\cdot\tangent_E)~\tangent_E)_{E\in\Eh}
  \big).
\end{equation}
The commutation property is the following, obtained by considering only the face components in the 3D formula \cite[Eq.~(3.33)]{Di-Pietro.Droniou:21}:
\begin{equation}\label{eq:tIrot:P2}
  \uGh(\Idis{h}v) = \tIrot{h}(\GRAD v)\qquad\forall v\in C^1(\overline{\Omega}).
\end{equation}

\subsubsection{Discrete scalar rotor and rotation reconstruction on $\uvec{\Theta}_T^k$}\label{sec:RT.PT.on.rotations}

Let a mesh element $T\in\Th$ be fixed. The local scalar rotor operator $\RT:\uvec{\Theta}_T^k\to\Poly{k}(T)$ is such that,
for all $\uvec{\eta}_T\in\uvec{\Theta}_T^k$,
\begin{equation}\label{eq:def.RT}
  \int_T\RT\uvec{\eta}_T q
  = \int_T\bvec{\eta}_{\cvec{R},T}\cdot\VROT q
  - \sum_{E\in\ET}\omega_{TE}\int_E(\bvec{\eta}_E\cdot\tangent_E)~q
  \qquad\forall q\in\Poly{k}(T).
\end{equation}
This operator enables the reconstruction of a discrete rotation $\PT:\uvec{\Theta}_T^k\to\vPoly{k}(T)$ defined such that, for all $\uvec{\eta}_T\in\uvec{\Theta}_T^k$ and all $(\bvec{\tau},q)\in\cRoly{k}(T)\times\Poly{k+1}(T)$,
\begin{equation}\label{eq:def.PT}
\int_T\PT\uvec{\eta}_T\cdot(\bvec{\tau} + \VROT q)
=
\int_T\bvec{\eta}_{\cvec{R},T}^\compl\cdot\bvec{\tau}
+ \int_T\RT\uvec{\eta}_T~q
+ \sum_{E\in\ET}\omega_{TE}\int_E(\bvec{\eta}_E\cdot\tangent_E)~q.
\end{equation}
The scalar rotor and rotation reconstructions correspond to the face curl and tangential face potential of the DDR method \cite[Eqs. (3.15) and (3.18)]{Di-Pietro.Droniou:21}. 
We note the following property \cite[Proposition 15]{Di-Pietro.Droniou:21}: For all $\uvec{\eta}_T\in\uvec{\Theta}_T^k$,
\begin{equation} \label{eq:projections.PT}
\Rproj{k-1}{T}(\PT\uvec{\eta}_T)=\bvec{\eta}_{\cvec{R},T}\quad\mbox{ and }\quad\cRproj{k}{T}(\PT\uvec{\eta}_T)=\bvec{\eta}_{\cvec{R},T}^\compl.
\end{equation}
Hence, for all $\bvec{\eta}\in H^1(T)^2$, we have $\Rproj{k-1}{T}\big[\PT(\Irot{T}\bvec{\eta})\big]=\Rproj{k-1}{T}\bvec{\eta}$ and $\cRproj{k-1}{T}\big[\PT(\Irot{T}\bvec{\eta})\big]=\cRproj{k-1}{T}\bvec{\eta}$ (where we have used $\cRoly{k-1}(T)\subset \cRoly{k}(T)$, see \eqref{eq:spaces.T}, to write $\cRproj{k-1}{T}=\cRproj{k-1}{T}\cRproj{k}{T}$).
Combining these relations with \eqref{eq:decomposition:Poly.ell.F} written for $\ell=k-1$ and \cite[Lemma 4]{Di-Pietro.Droniou:21}, we get
\begin{equation}\label{eq:vprojkmo.PT}
  \vlproj{k-1}{T}\big[\PT(\Irot{T}\bvec{\eta})\big]
  =\vlproj{k-1}{T}\bvec{\eta}\qquad\forall \bvec{\eta}\in H^1(T)^2.
\end{equation}

\subsubsection{Discrete symmetric gradient, divergence and stabilisation on $\uvec{\Theta}_T^k$}\label{sec:GsT.DT.sT.on.rotations}

The discretisation of the bilinear form \eqref{eq:bilinear.a} requires to define a discrete symmetric gradient (and divergence) on the discrete space of rotations. Since vectors in this space have polynomial components inside the elements and on the edges, a natural approach to define such discrete differential operators comes from the Hybrid High-Order (HHO) machinery \cite{Di-Pietro.Droniou:20}.
In what follows, we let a mesh element $T\in\Th$ be fixed.

\paragraph{Gradients and divergence.}
Let us define the local (vector-valued) HHO space, extension of $\uvec{\Theta}_T^k$ in which the element component is taken in the full polynomial space:
\begin{equation}\label{eq:def.Theta.HHO}
\uvec{\Theta}^k_{{\rm HHO},T}=\left\{\uvec{w}_T=(\bvec{w}_T,(\bvec{w}_E)_{E\in\ET})\,:\,\bvec{w}_T\in\vPoly{k}(T)\,,\quad\bvec{w}_E\in\vPoly{k}(E)\quad\forall E\in\ET\right\}.
\end{equation}
The discrete rotation enables the definition of the following embedding $\injHHO{T}:\uvec{\Theta}_T^k\to \uvec{\Theta}^k_{{\rm HHO},T}$:
\begin{equation}\label{eq:def.injHHO}
\injHHO{T} \uvec{\eta}_T \coloneq (\PT\uvec{\eta}_T,(\bvec{\eta}_E)_{E\in\ET})\qquad\forall\uvec{\eta}_T\in\uvec{\Theta}_T^k.
\end{equation}
Owing to \eqref{eq:projections.PT}, $\injHHO{T}$ is indeed a one-to-one mapping.

Using HHO techniques (see in particular \cite[Section 7.2.5]{Di-Pietro.Droniou:20}) on $\injHHO{T}\uvec{\eta}_T$, we can then design the local discrete gradients (standard and symmetric) and divergence of a discrete rotation $\uvec{\eta}_T\in \uvec{\Theta}_T^k$.
Specifically, this leads to defining the rotation gradient $\tGT:\uvec{\Theta}_T^k\to\tPoly{k}(T)$ such that, for all $\uvec{\eta}_T\in\uvec{\Theta}_T^k$,
\begin{equation}\label{eq:tGT}
  \int_T\tGT\uvec{\eta}_T:\tens{t}
  =-\int_T\PT\uvec{\eta}_T\cdot(\vDIV\tens{t})
  + \sum_{E\in\ET}\omega_{TE}\int_E\bvec{\eta}_E\cdot(\tens{t}\normal_E)
  \qquad\forall\tens{t}\in\tPoly{k}(T).
\end{equation}
The local symmetric gradient $\GsT:\uvec{\Theta}_T^k\to\tPoly[\rm s]{k}(T)$ and divergence $\DT:\uvec{\Theta}_T^k\to\Poly{k}(T)$ are obtained setting, for all $\uvec{\eta}_T\in\uvec{\Theta}_T^k$,
\[
\GsT\uvec{\eta}_T\coloneq\frac12\big[
  \tGT\uvec{\eta}_T + (\tGT\uvec{\eta}_T)^\intercal
  \big],\qquad
\DT\uvec{\eta}_T\coloneq\tr(\tGT\uvec{\eta}_T).
\]
In \eqref{eq:tGT}, since $\vDIV\tens{t}\in\vPoly{k-1}(T)$ we can replace $\PT\uvec{\eta}_T$ with $\vlproj{k-1}{T}(\PT\uvec{\eta}_T)$ and thus, using \eqref{eq:vprojkmo.PT} and following the techniques of \cite[Section 7.2.5]{Di-Pietro.Droniou:20}, we obtain the commutation formula $\tGT(\Irot{T}\bvec{\eta})=\vlproj{k}{T}(\GRAD\bvec{\eta})$ for all $\bvec{\eta}\in H^1(T)^2$; this shows that $\tGT$ (hence also $\GsT$ and $\DT$) has optimal approximation properties.

\paragraph{Stabilisation.}
As usual in numerical methods for polytopal meshes, the discrete counterpart of a bilinear form such as \eqref{eq:bilinear.a} involves a consistent component (here, based on $\GsT$), and a stabilisation term. In HHO methods, the local stabilisation bilinear forms are defined through the introduction of a higher-order reconstruction. For elasticity problems involving the discrete symmetric gradient, and accounting for the embedding \eqref{eq:def.injHHO}, this leads to defining $\pTHHO:\uvec{\Theta}_{T}\to\vPoly{k+1}(T)$ by: For all $\uvec{\eta}_T\in\uvec{\Theta}_{T}$ and all $\bvec{w}\in\vPoly{k+1}(T)$,
\begin{subequations}
\label{eq:def.pTHHO}
  \begin{align}
  \label{eq:def.pTHHO.1}
  &\int_T\GRADs \pTHHO\uvec{\eta}_T:\GRADs\bvec{w}
  =
  -\int_T \PT\uvec{\eta}_T\cdot\vDIV(\GRADs \bvec{w})
  + \sum_{E\in\ET}\int_E \bvec{\eta}_E\cdot(\GRADs\bvec{w}~\normal_{TE}), \\
  \label{eq:def.pTHHO.2}
  &\int_T\GRADss\pTHHO\uvec{\eta}_T=\frac12 \sum_{E\in\ET}\int_E (\bvec{\eta}_E\otimes\normal_{TE}-\normal_{TE}\otimes\bvec{\eta}_E)\,,\mbox{ and } \\
  \label{eq:def.pTHHO.3}
  &\int_T\pTHHO\uvec{\eta}_T=\int_T\PT\uvec{\eta}_T\quad\mbox{ if $k\ge 1$},\qquad
   \int_{\partial T}\pTHHO\uvec{\eta}_T=\sum_{E\in\ET}\int_E\bvec{\eta}_E\quad\mbox{ if $k=0$}.
  \end{align}
\end{subequations}
In a similar way as for $\tGT$ above, in \eqref{eq:def.pTHHO.1} the term $\PT\uvec{\eta}_T$ can be replaced with $\vlproj{k-1}{T}(\PT\uvec{\eta}_T)$ (because $\vDIV(\GRADs \bvec{w})\in \vPoly{k-1}(T)$).
Hence, using \eqref{eq:vprojkmo.PT} and the techniques of \cite[Section 7.2.5]{Di-Pietro.Droniou:20} we see that, for $k\ge 1$,
\begin{equation}\label{eq:pTHHO.strain}
\pTHHO(\Irot{T}\bvec{\eta}) = \sproj{k+1}\bvec{\eta}\qquad\forall\bvec{\eta}\in H^1(T)^2,
\end{equation}
where $\sproj{k+1}:H^1(T)^2\to\vPoly{k+1}(T)$ is the strain projector of degree $k+1$, see \cite[Section 7.2.2]{Di-Pietro.Droniou:20}. If $k=0$, the relation \eqref{eq:pTHHO.strain} is still verified with a modified version of the strain projector (still denoted by $\sproj{1}$), inspired by the modified elliptic projector of \cite[Section 5.1.2]{Di-Pietro.Droniou:20}, whose closure equation involves the average over $\partial T$ instead of the average over $T$; this modified strain projector has the same approximation properties as the standard strain projector.

The local stabilisation is then defined by:
\[
\mathrm{s}_T(\uvec{\tau}_T,\uvec{\eta}_T)=\sum_{E\in\ET}h_T^{-1}\int_E (\bvec{\delta}_{TE}^k-\bvec{\delta}_T^k)\uvec{\tau}_T \cdot (\bvec{\delta}_{TE}^k-\bvec{\delta}_T^k)\uvec{\eta}_T\quad \forall \uvec{\tau}_T,\uvec{\eta}_T\in\uvec{\Theta}_T^k,
\]
where the difference operators are such that, for all $\uvec{\eta}_T\in\uvec{\Theta}_T^k$ and $E\in\ET$,
\begin{equation}\label{eq:difference.operators}
  \bvec{\delta}^k_T\uvec{\eta}_T\coloneq \PT\big[\Irot{T}(\pTHHO\uvec{\eta}_T-\PT\uvec{\eta}_T)\big]\,,\quad
  \bvec{\delta}^k_{TE}\uvec{\eta}_T\coloneq \vlproj{k}{E}(\pTHHO\uvec{\eta}_T-\bvec{\eta}_E).
\end{equation}
Observing that $\PT\Irot{T}:H^1(T)^2\to\vPoly{k}(T)$ is a projector (see \cite[Eq.~(3.21)]{Di-Pietro.Droniou:21}) and using \eqref{eq:pTHHO.strain}, it can be checked that $\bvec{\delta}^k_T(\Irot{T}\bvec{\eta}) = \bvec{0}$ and $\bvec{\delta}^k_{TE}(\Irot{T}\bvec{\eta}) = \bvec{0}$ for all $E\in\ET$ and all $\bvec{\eta}\in\vPoly{k+1}(T)$. As a consequence, we have the following polynomial consistency property for $\mathrm{s}_T$:
\begin{equation}\label{eq:sT.polynomial}
  \mathrm{s}_T(\Irot{T}\bvec{\eta},\uvec{\xi}_T)=0\qquad
  \forall (\bvec{\eta},\uvec{\xi}_T)\in \vPoly{k+1}(T)\times\uvec{\Theta}_T^k.
\end{equation}

\begin{remark}[Original HHO stabilisation]
  In the original HHO stabilisation, the $L^2$-projector $\vlproj{k}{T}$ is used instead of $\PT\Irot{T}$ in the expression of $\bvec{\delta}^k_T$; see \eqref{eq:difference.operators}.
  The reason for using $\PT\Irot{T}$ here lies in the need to satisfy, for the interpolator $\Irot{T}$ on $\uvec{\Theta}_T^k$, the polynomial consistency \eqref{eq:sT.polynomial}.
  Note also that, in $\mathrm{s}_T$, the scaling factor $h_T^{-1}$ has been preferred over the original HHO scaling factor $h_E^{-1}$, as it is proved in \cite{Droniou.Yemm:21} to lead to a more robust discretisation in presence of small edges.
\end{remark}

Using the $L^2$-boundedness of $\PT\Irot{T}$ (stemming from the two-dimensional versions of \cite[Proposition 27 and Lemma 28]{Di-Pietro.Droniou:21}), the commutation property \eqref{eq:pTHHO.strain}, and the polynomial consistency \eqref{eq:sT.polynomial}, it is easy to reproduce, with our definitions of $\GsT$, $\pTHHO$ and $\mathrm{s}_T$, the standard HHO analysis of \cite[Section 7]{Di-Pietro.Droniou:20} and to obtain corresponding boundedness and consistency results (translated through $\injHHO{T}$).

\paragraph{Global operators.}
Global symmetric gradient, divergence, and higher-order reconstruction operators are obtained setting, for all $\uvec{\eta}_h\in\uvec{\Theta}_h^k$,
\[
\text{%
  $(\Gsh\uvec{\eta}_h)_{|T}\coloneq\GsT\uvec{\eta}_T$,
  $(\Dh\uvec{\eta}_h)_{|T}\coloneq\DT\uvec{\eta}_T$,
  and $(\phHHO\uvec{\eta}_h)_{|T}\coloneq\pTHHO\uvec{\eta}_T$ for all $T\in\Th$.
}
\]
Likewise, denoting by $\uvec{\Theta}_{{\rm HHO},h}^k$ the global HHO space obtained patching together the local spaces \eqref{eq:def.Theta.HHO} by enforcing the single-valuedness of the edge components, we define the global embedding $\injHHO{h}:\uvec{\Theta}_h^k\to\uvec{\Theta}_{{\rm HHO},h}^k$ by setting, for all $\uvec{\eta}_h \in\uvec{\Theta}_h^k$, $(\injHHO{h}\uvec{\eta}_h)_{|T}\coloneq\injHHO{T}\uvec{\eta}_T$ for all $T\in\Th$.
We also let $\mathrm{s}_h:\uvec{\Theta}_h^k\times \uvec{\Theta}_h^k\to\Real$ be the global stabilisation bilinear form such that
\[
\mathrm{s}_h(\uvec{\tau}_h,\uvec{\eta}_h)
\coloneq\sum_{T\in\Th}\mathrm{s}_T(\uvec{\tau}_T,\uvec{\eta}_T)\qquad
\forall(\uvec{\tau}_h,\uvec{\eta}_h)\in\uvec{\Theta}_h^k\times\uvec{\Theta}_h^k.
\]

\subsection{Discrete forms}

Based on the reconstructions introduced in the previous section, we define the discrete counterparts of the forms that appear in the weak formulation \eqref{eq:weak}.
Specifically, we let the bilinear form $\mathrm{A}_h:\left[\uvec{\Theta}_h^k\times\underline{U}_h^k\right]^2\to\Real$ and the linear form $\ell_h:\underline{U}_h^k\to\Real$ be such that, for all $(\uvec{\tau}_h,\underline{w}_h),(\uvec{\eta}_h,\underline{v}_h)\in\uvec{\Theta}_h^k\times\underline{U}_h^k$,
\begin{equation}\label{eq:Ah}
  \mathrm{A}_h((\uvec{\tau}_h,\underline{w}_h),(\uvec{\eta}_h,\underline{v}_h))
  \coloneq
  \mathrm{a}_h(\uvec{\tau}_h,\uvec{\eta}_h)
  + \mathrm{b}_h((\uvec{\tau}_h,\underline{w}_h),(\uvec{\eta}_h,\underline{v}_h)),\qquad
  \ell_h(\underline{v}_h)\coloneq\int_\Omega f~\PUh\underline{v}_h,
\end{equation}
where the bilinear forms $\mathrm{a}_h:\uvec{\Theta}_h^k\times\uvec{\Theta}_h^k\to\Real$ and
$\mathrm{b}_h:\left[\uvec{\Theta}_h^k\times\underline{U}_h^k\right]^2\to\Real$ are such that
\begin{equation}\label{eq:ah.bh}
  \begin{gathered}
    \mathrm{a}_h(\uvec{\tau}_h,\uvec{\eta}_h)
    \coloneq
    \beta_0\left(
    \int_\Omega\Gsh\uvec{\tau}_h:\Gsh\uvec{\eta}_h
    + \mathrm{s}_h(\uvec{\tau}_h,\uvec{\eta}_h)
    + \mathrm{j}_h(\uvec{\tau}_h,\uvec{\eta}_h)
    \right)
    + \beta_1\int_\Omega\Dh\uvec{\tau}_h~\Dh\uvec{\eta}_h,
    \\
    \mathrm{b}_h((\uvec{\tau}_h,\underline{w}_h),(\uvec{\eta}_h,\underline{v}_h))
    \coloneq
    \frac{\kappa}{t^2}(\uvec{\tau}_h - \uGh\underline{w}_h, \uvec{\eta}_h - \uGh\underline{v}_h)_{\bvec{\Theta},h}.
  \end{gathered}
\end{equation}
Here, $\mathrm{j}_h$ is an additional stabilisation term appearing only in the case $k=0$ and which penalises the jumps of higher-order reconstructions between elements:
\[
\mathrm{j}_h(\uvec{\tau}_h,\uvec{\eta}_h)
\coloneq\begin{cases}
0 &  \mbox{ if $k\ge 1$},\\
\displaystyle\sum_{E\in\Eh}h_E^{-1}\int_E [\phHHO[1]\uvec{\tau}_h]_E[\phHHO[1]\uvec{\eta}_h]_E & \mbox{ if $k=0$},
\end{cases}
\]
where, for any internal edge $E\in\Ehi$, if $T_1,T_2$ are the two elements (in an arbitrary but fixed order) on each side of $E$, we set $[\phHHO[1]\uvec{\tau}_h]_E\coloneq(\pToneHHO[1]\uvec{\tau}_{T_1})_{|E}-(\pTtwoHHO[1]\uvec{\tau}_{T_2})_{|E}$ while, for any boundary edge $E\in\Ehb\cap\ET$ for $T\in\Th$, $[\phHHO[1]\uvec{\tau}_h]_E\coloneq (\pTHHO[1]\uvec{\tau}_T)_{|E}$. We also introduced in \eqref{eq:ah.bh} the DDR $L^2$-product $(\cdot,\cdot)_{\bvec{\Theta},h}$ on $\uvec{\Theta}_h^k$ assembled from the following local contributions:
For all $\uvec{\tau}_T,\uvec{\eta}_T\in\uvec{\Theta}_T^k$,
\begin{equation}\label{eq:l2.prod.Theta.T}
  \begin{gathered}
    (\uvec{\tau}_T,\uvec{\eta}_T)_{\bvec{\Theta},T}
    \coloneq
    \int_T\PT\uvec{\tau}_T\cdot\PT\uvec{\eta}_T + \mathrm{S}_{\bvec{\Theta},T}(\uvec{\tau}_T,\uvec{\eta}_T)\\
    \text{
      with $\mathrm{S}_{\bvec{\Theta},T}(\uvec{\tau}_T,\uvec{\eta}_T)\coloneq\sum_{E\in\ET} h_E\int_E(\PT\uvec{\tau}_T-\bvec{\tau}_E)\cdot\tangent_E~(\PT\uvec{\eta}_T-\bvec{\eta}_E)\cdot\tangent_E$.
    }
  \end{gathered}
\end{equation}

\begin{remark}[Normal components of edge polynomials]\label{rem:normal.comp}
  A simple inspection of \eqref{eq:def.RT}, \eqref{eq:def.PT}, and \eqref{eq:l2.prod.Theta.T} shows that the normal components of edge unknowns do not enter the definition of $(\cdot,\cdot)_{\bvec{\Theta},h}$.
\end{remark}

\subsection{Discrete problem}

Define the following subspaces of $\uvec{\Theta}_h^k$ and $\underline{U}_h^k$ incorporating the clamped boundary condition:
\[
  \uvec{\Theta}_{h,0}^k
  \coloneq\Big\{
  \uvec{\eta}_h\in\uvec{\Theta}_h^k\st
  \text{%
    $\bvec{\eta}_E = \bvec{0}$ for all $E\in\Ehb$
  }
  \Big\},\qquad
  \underline{U}_{h,0}^k
  \coloneq\Big\{
  \underline{v}_h\in\underline{U}_h^k\st
  \text{$(v_{\Eh})_{|\partial\Omega} = 0$}
  \Big\}.
\]
The discrete problem reads:
Find $(\uvec{\theta}_h,\underline{u}_h)\in\uvec{\Theta}_{h,0}^k\times\underline{U}_{h,0}^k$ such that
\begin{equation}\label{eq:discrete}
  \mathrm{A}_h((\uvec{\theta}_h,\underline{u}_h),(\uvec{\eta}_h,\underline{v}_h))
  = \ell_h(\underline{v}_h)\qquad
  \forall(\uvec{\eta}_h,\underline{v}_h)\in\uvec{\Theta}_{h,0}^k\times\underline{U}_{h,0}^k.
\end{equation}


\section{Analysis}\label{sec:analysis}

Let
\begin{equation}\label{eq:mu}
  \mu\coloneq\min(\kappa,\beta_0).
\end{equation}
Throughout the rest of the paper, we use $a\lesssim b$ as a shorthand notation for the inequality $a\le Cb$ with multiplicative constant $C$ that possibly depends on $\Omega$, the mesh regularity, and on the polynomial degree, but not on $\beta_0$, $\beta_1$, $\kappa$, $\mu$, $t$, or $h$ and, for local inequalities, on the mesh element or edge.

\subsection{Discrete norm and stability}

We define the discrete seminorm on $\uvec{\Theta}_h^k\times\underline{U}_h^k$ such that, for all $(\uvec{\eta}_h,\underline{v}_h)\in\uvec{\Theta}_h^k\times\underline{U}_h^k$,
\begin{multline}\label{eq:norm.en.h}
  \norm[\bvec{\Theta}\times U,h]{(\uvec{\eta}_h,\underline{v}_h)}
  \coloneq\bigg[
    \beta_0\left(
    \norm[L^2(\Omega)^{2\times 2}]{\Gsh\uvec{\eta}_h}^2
    + \seminorm[{\rm s,j},h]{\uvec{\eta}_h}^2
    \right)
    + \beta_1\norm[L^2(\Omega)]{\Dh\uvec{\eta}_h}^2
    \\
    + \frac{\kappa}{t^2}\norm[\bvec{\Theta},h]{\uvec{\eta}_h-\uGh\underline{v}_h}^2
    + \mu\left(
    \norm[\bvec{\Theta},h]{\uvec{\eta}_h}^2
    + \norm[\bvec{\Theta},h]{\uGh\underline{v}_h}^2
    \right)
    \bigg]^{\frac12},
\end{multline}
where $\norm[\bvec{\Theta},h]{{\cdot}}$ and $\seminorm[{\rm s,j},h]{{\cdot}}$ denote the seminorms respectively induced by $(\cdot,\cdot)_{\bvec{\Theta},h}$ and $\mathrm{s}_h+\mathrm{j}_h$ on $\uvec{\Theta}_h^k$.
Using, respectively, the discrete Korn and Korn--Poincar\'e inequalities \cite[Lemma 7.24 and Eq.~(7.73)]{Di-Pietro.Droniou:20} (see also \cite[Lemma 7.42]{Di-Pietro.Droniou:20} in the case $k=0$) and the fact that $h_E\lesssim 1$ for the terms composing the norm in the left-hand side (see \eqref{eq:l2.prod.Theta.T} for the corresponding local contribution), we readily obtain
  \begin{equation}\label{eq:korn}
    \norm[\bvec{\Theta},h]{\uvec{\eta}_h}\lesssim\left(
    \norm[L^2(\Omega)^{2\times 2}]{\Gsh\uvec{\eta}_h}^2
    + \seminorm[{\rm s,j},h]{\uvec{\eta}_h}^2
    \right)^{\frac12}\qquad \forall\uvec{\eta}_h\in\uvec{\Theta}_{h,0}^k.
  \end{equation}
  Together with the Poincar\'e inequality for $\uGh$ in $\underline{U}_{h,0}^k$, whose proof can be obtained using arguments similar to \cite[Theorem 31]{Di-Pietro.Droniou:21} (leveraging the Poincar\'e inequality with zero boundary condition stated in \cite[Lemma 2.15]{Di-Pietro.Droniou:20}), \eqref{eq:korn} proves that the energy seminorm $\norm[\bvec{\Theta}\times U,h]{{\cdot}}$ is actually a norm on $\uvec{\Theta}_{h,0}^k\times\underline{U}_{h,0}^k$.
We can now establish the coercivity of $\mathrm{A}_h$ with respect to this norm.

\begin{lemma}[Coercivity]
  For all $(\uvec{\eta}_h,\underline{v}_h)\in\uvec{\Theta}_{h,0}^k\times\underline{U}_{h,0}^k$, it holds
  \begin{equation}\label{eq:coercivity}
    \norm[\bvec{\Theta}\times U,h]{(\uvec{\eta}_h,\underline{v}_h)}^2
    \lesssim\mathrm{A}_h((\uvec{\eta}_h,\underline{v}_h),(\uvec{\eta}_h,\underline{v}_h)).
  \end{equation}
\end{lemma}
\begin{proof}
  By the definitions \eqref{eq:Ah} of $\mathrm{A}_h$ and \eqref{eq:ah.bh} of $\mathrm{a}_h$ and $\mathrm{b}_h$, we have
  \begin{equation}\label{eq:coercivity:1}
    \beta_0\left(
    \norm[L^2(\Omega)^{2\times 2}]{\Gsh\uvec{\eta}_h}^2
    {+} \seminorm[{\rm s,j},h]{\uvec{\eta}_h}^2
    \right)
    + \beta_1\norm[L^2(\Omega)]{\Dh\uvec{\eta}_h}^2
    {+} \frac{\kappa}{t^2}\norm[\bvec{\Theta},h]{\uvec{\eta}_h-\uGh\underline{v}_h}^2
    = \mathrm{A}_h((\uvec{\eta}_h,\underline{v}_h),(\uvec{\eta}_h,\underline{v}_h)).
  \end{equation}
  We next write, using a triangle inequality,
  \begin{equation}\label{eq:coercivity:2}
    \begin{aligned}
      \norm[\bvec{\Theta},h]{\uGh\underline{v}_h}^2+\norm[\bvec{\Theta},h]{\uvec{\eta}_h}^2
      &\le2\norm[\bvec{\Theta},h]{\uvec{\eta}_h - \uGh\underline{v}_h}^2
      + 3\norm[\bvec{\Theta},h]{\uvec{\eta}_h}^2
      \\
      &\le 2\kappa^{-1}\frac{\kappa}{t^2}\norm[\bvec{\Theta},h]{\uvec{\eta}_h - \uGh\underline{v}_h}^2
      + 3\beta_0^{-1}\beta_0\left(\norm[L^2(\Omega)^{2\times 2}]{\Gsh\uvec{\eta}_h}^2+\seminorm[{\rm s,j},h]{\uvec{\eta}_h}^2\right)
      \\
      &\lesssim\mu^{-1}\mathrm{A}_h((\uvec{\eta}_h,\underline{v}_h),(\uvec{\eta}_h,\underline{v}_h)),
    \end{aligned}
  \end{equation}
  where we have used the fact that $t<1$ along with the discrete Korn inequality \eqref{eq:korn} to pass to the second line
  and \eqref{eq:coercivity:1} together with the definition \eqref{eq:mu} of $\mu$ to conclude.
  The proof is completed by combining \eqref{eq:coercivity:1} and \eqref{eq:coercivity:2} with the definition of $\norm[\bvec{\Theta}\times U,h]{{\cdot}}$.
\end{proof}

\subsection{Error estimates}\label{sec:results}

The regularity assumptions in the error estimates are expressed in terms of the broken Sobolev spaces
\[
H^s(\Th)\coloneq\left\{
v\in L^2(\Omega)\st\text{$v_{|T}\in H^s(T)$ for all $T\in\Th$}
\right\}.
\]
The first error estimate is for an arbitrary polynomial degree $k$.

\begin{theorem}[Error estimate for arbitrary $k$]\label{th:error.est}
  Denote by $(\bvec{\eta},u)\in\bvec{\Theta}\times U$ and $(\uvec{\eta}_h,\underline{u}_h)\in\uvec{\Theta}_{h,0}^k\times\underline{U}_{h,0}^k$ the solutions to problems \eqref{eq:weak} and \eqref{eq:discrete}, respectively.
  We assume the additional regularity
  $u\in C^1(\overline{\Omega})\cap H^{k+2}(\Th)$ for the displacement and
  $\bvec{\theta}\in H^1(\Omega)^2\cap H^{k+2}(\Th)^2$ for the rotation. Then, it holds
  \begin{equation}\label{eq:err.est}  
    \norm[\bvec{\Theta}\times U,h]{(\uvec{\theta}_h - \Irot{h}\bvec{\theta}, \underline{u}_h - \Idis{h}u)}
    \lesssim h^{k+1}\left(
      \beta_0^{-\frac12}(\beta_0 + \beta_1)\seminorm[H^{k+2}(\Th)^2]{\bvec{\theta}}
      + 
      \mu^{-\frac12}\seminorm[H^{k+1}(\Th)^2]{\bvec{\gamma}}
      \right).
  \end{equation}
\end{theorem}

\begin{remark}[Regularity of the shear strain {$\bvec{\gamma}$}]
  Under the regularity assumptions on $u$ and $\bvec{\theta}$ in the theorem, the shear strain defined by \eqref{eq:rm.def.gamma} satisfies $\bvec{\gamma}\in H^1(\Omega)^2\cap H^{k+1}(\Th)^2$.
\end{remark}

\begin{proof}
See Section \ref{sec:proof.th.error.est}.
\end{proof}
The bound \eqref{eq:err.est} shows that the DDR scheme achieves as expected a high-order of accuracy, when the solution is smooth enough and $t$ is not too small. When $t\to 0$ the higher derivatives of the shear strain $\bvec{\gamma}$ are known to explode; typically, $\seminorm[H^{k+1}(\Th)^2]{\bvec{\gamma}}$ grows as $t^{-k-1}$, as explained in \cite[Theorem 2.1 and following remarks]{Arnold.Falk:89}.
Thus, even though $t$ does not explicitly appear in the right-hand side of \eqref{eq:err.est}, the dependency of this right-hand side on higher derivatives of the solution means that this estimate is not locking-free. Such a dependency is unavoidable for high-order schemes (see, e.g., \cite{Arnold.Brezzi.ea:05} in the case of continuous/discontinuous Galerkin schemes).
However, for $k=0$, one could expect a better estimate than \eqref{eq:err.est} in which $\seminorm[H^1(\Th)]{\bvec{\gamma}}$ is multiplied by $t$ as in \cite{Arnold.Brezzi.ea:05,Chinosi.Lovadina.ea:06,Lovadina:05,Beirao-da-Veiga.Mora.ea:19}; this ensures that the method is locking-free at least if $\Omega$ is convex since, on such domains, $t\seminorm[H^1(\Omega)]{\bvec{\gamma}}$ remains bounded as $t\to 0$. Such an error estimate is stated in the next theorem. Note that, contrary to most analyses in the aforementioned references and others (a notable exception being \cite{Beirao-da-Veiga.Mora.ea:19}), the proof of the following estimate does not use a Helmoltz decomposition of the shear strain.

\begin{theorem}[Locking-free error estimate for $k=0$]\label{th:error.est.k0}
Assume the hypotheses of Theorem \ref{th:error.est}, and that $k=0$. Then, it holds
  \begin{multline}\label{eq:err.est.k0}  
    \norm[\bvec{\Theta}\times U,h]{(\uvec{\theta}_h - \Irotz{h}\bvec{\theta}, \underline{u}_h - \Idisz{h}u)}
    \\\lesssim 
    h\left(\beta_0^{-\frac12}(\beta_0 + \beta_1)\seminorm[H^{2}(\Th)^2]{\bvec{\theta}}
    +\kappa^{-\frac12}t\seminorm[H^1(\Th)^2]{\bvec{\gamma}}+\beta_0^{-\frac12}\norm[L^2(\Omega)^2]{\bvec{\gamma}}
    +\mu^{-\frac12}\norm[L^2(\Omega)]{f}
    \right).
  \end{multline}
\end{theorem}
\begin{proof}
See Section \ref{sec:proof.th.error.est.k0}.
\end{proof}

\begin{remark}[Locking-free property]
  If $\Omega$ is convex, all terms in the right-hand side of \eqref{eq:err.est.k0} are bounded independently of $t$ \cite[Theorem 2.1]{Arnold.Falk:89}.
    The techniques used to prove \eqref{eq:err.est.k0} can be extended (at the price of some technicalities) to arbitrary values of $k$ to replace, in the right-hand side of \eqref{eq:err.est}, the term $\seminorm[H^{k+1}(\Th)^2]{\bvec{\gamma}}$ with $t\seminorm[H^{k+1}(\Th)^2]{\bvec{\gamma}} + \seminorm[H^k(\Th)^2]{\bvec{\gamma}} + \seminorm[H^k(\Th)]{f}$ in the spirit of \cite[Remark 4.3]{Arnold.Brezzi.ea:05}.
    However, since a bound independent of $t$ for the quantity $t\seminorm[H^{k+1}(\Th)^2]{\bvec{\gamma}} + \seminorm[H^k(\Th)^2]{\bvec{\gamma}} + \seminorm[H^k(\Th)]{f}$ can only be established for $k=0$, this would not yield complete robustness of the estimate \eqref{eq:err.est} for $k\ge 1$.
    For this reason, and also to make the exposition less technical, we have decided to state two separate estimates.
\end{remark}

\subsection{Proof of the arbitrary-order error estimate}\label{sec:proof.th.error.est}

\begin{proof}[Proof of Theorem \ref{th:error.est}]
  \underline{1. \emph{Basic error estimate}.}
  Combining the coercivity \eqref{eq:coercivity} of $\mathrm{A}_h$ with the Third Strang Lemma \cite[Theorem 10]{Di-Pietro.Droniou:18}, we obtain the following basic error estimate:
  \begin{equation}\label{eq:basic.err.est}
    \norm[\bvec{\Theta}\times U,h]{(\uvec{\theta}_h - \Irot{h}\bvec{\theta}, \underline{u}_h - \Idis{h}u)}
    \lesssim\sup_{(\uvec{\eta}_h,\underline{v}_h)\in\uvec{\Theta}_{h,0}^k\times\underline{U}_{h,0}^k\setminus\{(\uvec{0},\underline{0})\}}\frac{%
      \mathcal{E}_h((\bvec{\theta},u);(\uvec{\eta}_h,\underline{v}_h))
    }{\norm[\bvec{\Theta}\times U,h]{(\uvec{\eta}_h,\underline{v}_h)}},
  \end{equation}
  where the consistency error linear form $\mathcal{E}_h((\bvec{\theta},u);\cdot)$ is such that, for all $(\uvec{\eta}_h,\underline{v}_h)\in\uvec{\Theta}_{h,0}^k\times\underline{U}_{h,0}^k$,
  \begin{equation}\label{eq:consistency.error}
    \mathcal{E}_h((\bvec{\theta},u);(\uvec{\eta}_h,\underline{v}_h))
    \coloneq
    \ell_h(\underline{v}_h) - \mathrm{A}_h((\Irot{h}\bvec{\theta},\Idis{h}u),(\uvec{\eta}_h,\underline{v}_h)).
  \end{equation}
  \\
  \underline{2. \emph{Reformulation of the consistency error}.}
  To prove \eqref{eq:err.est}, we need to estimate the dual norm of the consistency error, which corresponds to the right-hand side of \eqref{eq:basic.err.est}. 
    We first recast $\mathrm{b}_h$. Recall the definition \eqref{eq:def.tIrot} of the modified interpolator $\tIrot{h}$ and notice that, by Remark \ref{rem:normal.comp}, it holds
    \begin{equation}\label{eq:tIrot:P1}
      (\Irot{h}\bvec{\eta},\uvec{\tau}_h)_{\bvec{\Theta},h}
      = (\tIrot{h}\bvec{\eta},\uvec{\tau}_h)_{\bvec{\Theta},h}
      \qquad\forall(\bvec{\eta},\uvec{\tau}_h)\in H^1(\Omega)^2\times\uvec{\Theta}_{h}^k.
    \end{equation}
  We can then write
  \begin{align*}
      \mathrm{b}_h((\Irot{h}\bvec{\theta},\Idis{h}u),(\uvec{\eta}_h,\underline{v}_h))
      &= \frac{\kappa}{t^2}(\tIrot{h}\bvec{\theta} - \uGh(\Idis{h}u), \uvec{\eta}_h - \uGh\underline{v}_h)_{\bvec{\Theta},h}    
      \\
      &= \frac{\kappa}{t^2}(\tIrot{h}(\bvec{\theta} - \GRAD u), \uvec{\eta}_h - \uGh\underline{v}_h)_{\bvec{\Theta},h}
      \\
      &= (\Irot{h}\bvec{\gamma}, \uGh\underline{v}_h - \uvec{\eta}_h)_{\bvec{\Theta},h},
  \end{align*}
  where we have used the definition \eqref{eq:ah.bh} of $\mathrm{b}_h$ along with \eqref{eq:tIrot:P1} in the first line,
  the key commutation property \eqref{eq:tIrot:P2} to pass to the second line,
  and the definition \eqref{eq:rm.def.gamma} of the shear strain $\bvec{\gamma}$ followed by \eqref{eq:tIrot:P1} to conclude. Expanding the inner product $(\cdot,\cdot)_{\bvec{\Theta},h}$ according to its definition from the local products \eqref{eq:l2.prod.Theta.T} and using the relation $\PT\uGT=\GT$ (see \cite[Proposition 15]{Di-Pietro.Droniou:21}), we infer
  \begin{equation}
      \mathrm{b}_h((\Irot{h}\bvec{\theta},\Idis{h}u),(\uvec{\eta}_h,\underline{v}_h))\\
      = \sum_{T\in\Th} \int_T\bvec{\gamma}\cdot(\GT\underline{v}_T - \PT\uvec{\eta}_T) -\term
  \label{eq:bh:interpolate}
  \end{equation}
  with
  \[
  \term\coloneq
  \sum_{T\in\Th} \int_T[\bvec{\gamma}-\PT(\Irot{T}\bvec{\gamma})]\cdot\PT(\uGT\underline{v}_T - \uvec{\eta}_T)
      + \sum_{T\in\Th}\mathrm{S}_{\bvec{\Theta},T}(\Irot{T}\bvec{\gamma},\uGT\underline{v}_T - \uvec{\eta}_T).
  \]
  
  Accounting for the definition of the material tensor $\tens{C}$, we also have, for all $\uvec{\tau}_h,\uvec{\eta}_h\in\uvec{\Theta}_h^k$,
  \begin{equation}\label{eq:ah.bis}
    \mathrm{a}_h(\uvec{\tau}_h,\uvec{\eta}_h)
    = \int_\Omega\tens{C}\Gsh\uvec{\tau}_h:\Gsh\uvec{\eta}_h
    + \beta_0\mathrm{s}_h(\uvec{\tau}_h,\uvec{\eta}_h)+ \beta_0\mathrm{j}_h(\uvec{\tau}_h,\uvec{\eta}_h).
  \end{equation}

  Recalling the definitions \eqref{eq:Ah} of $\mathrm{A}_h$ and $\ell_h$, the relations $\bvec{\gamma}=-\vDIV(\tens{C}\GRADs\bvec{\theta})$ and $f=-\DIV\bvec{\gamma}$ (see \eqref{eq:rm.balance1} and \eqref{eq:rm.balance2}) along with \eqref{eq:bh:interpolate} and \eqref{eq:ah.bis} shows that the consistency error \eqref{eq:consistency.error} can be recast as
  \begin{equation}\label{eq:Eh.decomposition}
    \mathcal{E}_h((\bvec{\theta},u);(\uvec{\eta}_h,\underline{v}_h))
    =\mathcal{E}_{\GRAD,h}(\bvec{\gamma};\underline{v}_h)+\term
    + \mathcal{E}_{\GRADs,h}(\tens{C}\GRADs\bvec{\theta};\uvec{\eta}_h),
  \end{equation}
  where the adjoint consistency errors for the gradient on $\underline{U}_h^k$ and for the symmetric gradient on $\uvec{\Theta}_h^k$ are defined as
  \begin{align*}
  \mathcal{E}_{\GRAD,h}(\bvec{\gamma};\underline{v}_h)
  \coloneq{}& -\int_\Omega \DIV\bvec{\gamma}~\PUh\underline{v}_h
  - \sum_{T\in\Th}\int_T\bvec{\gamma}\cdot\GT\underline{v}_T,
  \\
  \mathcal{E}_{\GRADs,h}(\tens{C}\GRADs\bvec{\theta};\uvec{\eta}_h)
    \coloneq{}&
  - \sum_{T\in\Th}\int_T\vDIV(\tens{C}\GRADs\bvec{\theta})\cdot\PT\uvec{\eta}_T
  - \int_\Omega\tens{C}\Gsh(\Irot{h}\bvec{\theta}):\Gsh\uvec{\eta}_h\\
  &- \beta_0\mathrm{s}_h(\Irot{h}\bvec{\theta},\uvec{\eta}_h)- \beta_0\mathrm{j}_h(\Irot{h}\bvec{\theta},\uvec{\eta}_h).
  \end{align*}
  \\
  \underline{3. \emph{Bound on the consistency error}.} 
  To deal with $\mathcal E_{\GRADs,h}$, we use the estimate in \cite[Lemma 7.27]{Di-Pietro.Droniou:20} (and \cite[Lemma 7.43]{Di-Pietro.Droniou:20} if $k=0$) which, in the present context, yields
  \begin{align}
    |\mathcal{E}_{\GRADs,h}(\tens{C}\GRADs\bvec{\theta};\uvec{\eta}_h)|&\lesssim h^{k+1}(\beta_0 + \beta_1)\seminorm[H^{k+2}(\Th)^2]{\bvec{\theta}}\left(
    \norm[L^2(\Omega)^{2\times 2}]{\Gsh\uvec{\eta}_h}^2
    + \seminorm[{\rm s,j},h]{\uvec{\eta}_h}^2
    \right)^{\frac12}
    \nonumber\\
    &\lesssim
    h^{k+1}\beta_0^{-\frac12}(\beta_0 + \beta_1)\seminorm[H^{k+2}(\Th)^2]{\bvec{\theta}}
    \norm[\bvec{\Theta}\times U,h]{(\uvec{\eta}_h,\underline{v}_h)},
    \label{est:Egrads}
  \end{align}
  where the conclusion follows from the definition \eqref{eq:norm.en.h} of the discrete norm.

  The term $\term$ is estimated using Cauchy--Schwarz inequalities:
  \begin{align*}
    |\term|\le{}&
    \sum_{T\in\Th}\norm[L^2(T)^2]{\bvec{\gamma}-\PT(\Irot{T}\bvec{\gamma})}~\norm[L^2(T)^2]{\PT(\uGT\underline{v}_T-\uvec{\eta}_T)} \\
  &+ \sum_{T\in\Th}\mathrm{S}_{\bvec{\Theta},T}(\Irot{T}\bvec{\gamma},\Irot{T}\bvec{\gamma})^{\frac12}~\mathrm{S}_{\bvec{\Theta},T}(\uGT\underline{v}_T-\uvec{\eta}_T,\uGT\underline{v}_T-\uvec{\eta}_T)^{\frac12}\\
    \lesssim{}& \sum_{T\in\Th} h^{k+1}\seminorm[H^{k+1}(T)^2]{\bvec{\gamma}}\left(\norm[L^2(T)^2]{\PT(\uGT\underline{v}_T-\uvec{\eta}_T)}+\mathrm{S}_{\bvec{\Theta},T}(\uGT\underline{v}_T-\uvec{\eta}_T,\uGT\underline{v}_T-\uvec{\eta}_T)^{\frac12}\right)\\
    \lesssim{}& h^{k+1}\seminorm[H^{k+1}(\Th)^2]{\bvec{\gamma}}\norm[\bvec{\Theta},h]{\uGh\underline{v}_h-\uvec{\eta}_h},
  \end{align*}
  where we have used, in the second line, the consistency properties of $\PT\Irot{T}$ and $\mathrm{S}_{\bvec{\Theta},T}$ (two-dimensional versions of \cite[Eqs. (6.3) and (6.9)]{Di-Pietro.Droniou:21}, see Remark \ref{rem:cons.2D} below), and the conclusion follows from Cauchy--Schwarz inequalities on the sum and the definition of the norm $\norm[\bvec{\Theta},h]{{\cdot}}$.
  Using the definition \eqref{eq:norm.en.h} of the discrete norm, we infer
  \begin{equation}\label{est:T1}
  |\term|\lesssim h^{k+1} \kappa^{-\frac12} t\seminorm[H^{k+1}(\Th)^2]{\bvec{\gamma}}\norm[\bvec{\Theta}\times U,h]{(\uvec{\eta}_h,\underline{v}_h)}.
  \end{equation}
  
  The estimate of $\mathcal E_{\GRAD,h}$ follows proceeding as in the proof of \cite[Theorem 39]{Di-Pietro.Droniou:21}, with straightforward modifications to account for the different boundary conditions, fewer (and simpler) terms to track, and accounting for Remark \ref{rem:cons.2D}:  
  \begin{equation}\label{est:Egrad}
    |\mathcal E_{\GRAD,h}(\bvec{\gamma};\underline{v}_h)|
    \lesssim h^{k+1}\seminorm[H^{k+1}(\Th)^2]{\bvec{\gamma}}
    \norm[\bvec{\Theta},h]{\uGh\underline{v}_h}
    \le h^{k+1}\mu^{-\frac12}\seminorm[H^{k+1}(\Th)^2]{\bvec{\gamma}}
    \norm[\bvec{\Theta}\times U,h]{(\uvec{\eta}_h,\underline{v}_h)}.
  \end{equation}
  \\
  \underline{4. \emph{Conclusion}.}
  Plugging the estimates \eqref{est:Egrads}, \eqref{est:T1} and \eqref{est:Egrad} into \eqref{eq:Eh.decomposition}, we arrive at
  \begin{multline*}
    |\mathcal{E}_h((\bvec{\theta},u);(\uvec{\eta}_h,\underline{v}_h))|
    \\
    \lesssim h^{k+1}\left(
      \beta_0^{-\frac12}(\beta_0 + \beta_1)\seminorm[H^{k+2}(\Th)^2]{\bvec{\theta}}
      + \kappa^{-\frac12} t \seminorm[H^{k+1}(\Th)^2]{\bvec{\gamma}}
      + \mu^{-\frac12}\seminorm[H^{k+1}(\Th)^2]{\bvec{\gamma}}
      \right)
    \norm[\bvec{\Theta}\times U,h]{(\uvec{\eta}_h,\underline{v}_h)}.
  \end{multline*}
  The estimate \eqref{eq:err.est} follows using this bound in \eqref{eq:basic.err.est} and recalling that $\mu\le \kappa$ and $t\le 1$.
\end{proof}

\begin{remark}[Norms of {$\bvec{\gamma}$} in the estimates]\label{rem:cons.2D}
  In \cite{Di-Pietro.Droniou:21}, the estimates mentioned above (Theorem 39 and Section 6.1) involve, in the case $k=0$, weighted $H^2$-seminorms of $\bvec{\gamma}$. This is because these estimates are stated in three dimensions, in which interpolating a function on $\uvec{\Theta}_h^k$ requires a higher minimal regularity (to ensure traces along the edges are well-defined). In two dimensions, the local interpolator obtained restricting \eqref{eq:Irot.h} to $T$ is well-defined on $H^1(T)^2$, and the seminorm in this space is sufficient to state the consistency estimates for $k=0$.
\end{remark}

\subsection{Proof of the low-order locking-free error estimate}\label{sec:proof.th.error.est.k0}

The proof of Theorem \ref{th:error.est.k0} relies on liftings of elements in $\underline{U}_T^0$ and $\uvec{\Theta}_T^0$, for each $T\in\Th$. The assumption on the mesh yields a conforming simplicial subdivision $\ST$ of $T$ that is shape regular (with the same regularity parameter as in the mesh regularity assumption); actually, by \cite[Assumption 7.6]{Di-Pietro.Droniou:20} each $T\in\Th$ is star-shaped with respect to every point in a ball of radius $\varrho h_T$, where $\varrho$ is the mesh regularity parameter, so $\ST$ can be constructed by adding only one vertex (the center of that ball) in $T$ and creating the triangles between this vertex and the edges of $T$.
The proof given here, however, applies also to elements that are possibly not star-shaped.

The coordinate $\bvec{x}_V$ of any vertex $V$ of $\ST$ can be written as a convex combination of the coordinates of the vertices $\VT$ of $T$:
\begin{equation}\label{eq:lambda}
  \bvec{x}_V=\sum_{W\in\VT}\lambda_{V,W}\bvec{x}_{W}\,,\quad\mbox{ with $\lambda_{V,W}\ge 0$ and }\sum_{W\in\VT}\lambda_{V,W}=1
\end{equation}
(this includes the vertices $V\in\VT$, in which case we choose $\lambda_{V,V}=1$ and $\lambda_{V,W}=0$ if $W\not=V$).
Denoting by $\Poly[\rm c]{1}(\ST)$ the space of $H^1(T)$-conforming piecewise $\Poly{1}$ functions on $\ST$, for all $\underline{z}_T=z_{\ET}\in\underline{U}_T^0$ we define $\widetilde{z}_T\in\Poly[\rm c]{1}(\ST)$ such that 
\[
\widetilde{z}_T(\bvec{x}_V)=\sum_{W\in\VT}\lambda_{V,W}\,z_{\ET}(\bvec{x}_{W}).
\]
This construction is linearly exact, that is
\begin{equation}\label{eq:tilde.linearly.exact}
\widetilde{\Idisz{T}\phi_T}=\phi_T\qquad\forall\phi_T\in\Poly{1}(T).
\end{equation}
The next lemma, whose proof is postponed to the end of the section, states useful properties of the lifting $\underline{U}_T^0\ni\underline{z}_T\mapsto \widetilde{z}_T\in \Poly[\rm c]{1}(\ST)$.

\begin{lemma}[Properties of the lifting on $\underline{U}_T^0$]\label{lem:lift.U}
The following properties hold: For all $\underline{z}_T\in\underline{U}_T^0$,
\begin{align}
\label{lem:lift.U.proj} 
\vlproj{0}{T}(\GRAD\widetilde{z}_T)={}&\GT[0]\underline{z}_T,\\
\label{lem:lift.U.norm}
\norm[L^2(T)^2]{\GRAD\widetilde{z}_T}\lesssim{}& \norm[\bvec{\Theta},T]{\uGT[0]\underline{z}_T},\\
\label{lem:lift.U.diff} 
\norm[L^2(T)]{\widetilde{z}_T-\PUT[1]\underline{z}_T}\lesssim{}& h_T\norm[\bvec{\Theta},T]{\uGT[0]\underline{z}_T},
\end{align}
where $\norm[\bvec{\Theta},T]{{\cdot}}$ is the local seminorm induced by the product \eqref{eq:l2.prod.Theta.T} on $\uvec{\Theta}_T^0$.
Moreover, if $\underline{z}_h\in \underline{U}_{h,0}^0$ and $\widetilde{z}_h$ is defined such that $(\widetilde{z}_h)_{|T}=\widetilde{z}_T$ for all $T\in\Th$, then $\widetilde{z}_h\in H^1_0(\Omega)$.
\end{lemma}

We now define the lifting on $\uvec{\Theta}_T^0$. For any $\uvec{\eta}_T=(\bvec{\eta}_E)_{E\in\ET}\in\uvec{\Theta}_T^0$, let $\uvec{\eta}_T^\flat=((\bvec{\eta}_E\cdot\tangent_E)\tangent_E)_{E\in\ET}$ be the vector comprising only the tangential components to the edges. Since $(\cdot,\cdot)_{\bvec{\Theta},T}$ is an inner product when only these components are considered, we can write a unique decomposition
\[
\uvec{\eta}_T^\flat=\uGT[0]\underline{w}_T + \uvec{\kappa}_T\mbox{ with $\underline{w}_T\in\underline{U}_T^0$ and $\uvec{\kappa}_T\perp \uGT[0]\underline{U}_T^0$},
\]
the orthogonality being understood for $(\cdot,\cdot)_{\bvec{\Theta},T}$. We then set
\begin{equation*}
  \tvec{\eta}_T\coloneq\GRAD\widetilde{w}_T + \PT[0]\uvec{\kappa}_T.
\end{equation*}
The proof of the properties of the lifting $\uvec{\Theta}_T^0\ni\uvec{\eta}_T\mapsto \tvec{\eta}_T\in L^2(T)^2$ stated in the following lemma is postponed to the end of the section.

\begin{lemma}[Properties of the lifting on $\uvec{\Theta}_T^0$]\label{lem:lift.Theta}
The following properties hold: For all $\uvec{\eta}_T\in\uvec{\Theta}_T^0$, 
\begin{align}
\label{lem:lift.Theta.proj} 
\vlproj{0}{T}\tvec{\eta}_T={}&\PT[0]\uvec{\eta}_T,\\
\label{lem:lift.Theta.norm}
\norm[L^2(T)^2]{\GRAD\widetilde{v}_T-\tvec{\eta}_T}\lesssim{}&\norm[\bvec{\Theta},T]{\uGT[0]\underline{v}_T-\uvec{\eta}_T}\qquad
\forall \underline{v}_T\in\underline{U}_T^0.
\end{align}
Moreover, for all $\uvec{\eta}_h\in \uvec{\Theta}_{h,0}^0$,
\begin{equation}\label{lem:lift.Theta.diff}
\left(\sum_{T\in\Th}\norm[L^2(T)^2]{\tvec{\eta}_T-\PT[0]\uvec{\eta}_T}^2\right)^{\frac12}
\lesssim h\left(\norm[L^2(\Omega)^{2\times 2}]{\Gsh[0]\uvec{\eta}_h}^2+\seminorm[{\rm s,j},h]{\uvec{\eta}_h}^2\right)^{\frac12}.
\end{equation}
\end{lemma}

We are now ready to prove Theorem \ref{th:error.est.k0}.

\begin{proof}[Proof of Theorem \ref{th:error.est.k0}]
Given the basic error estimate \eqref{eq:basic.err.est}, we only have to find a proper upper bound of the consistency error.
We consider the first term in the expression \eqref{eq:bh:interpolate} of $\mathrm{b}_h((\Irot{h}\bvec{\theta},\Idis{h}u),(\uvec{\eta}_h,\underline{v}_h))$. Owing to \eqref{lem:lift.U.proj} and \eqref{lem:lift.Theta.proj} we have
\begin{align*}
\int_T\bvec{\gamma}\cdot(\GT[0]\underline{v}_T - \PT[0]\uvec{\eta}_T)={}&\int_T \bvec{\gamma}\cdot \vlproj{0}{T}(\GRAD\widetilde{v}_T-\tvec{\eta}_T)
=\int_T \vlproj{0}{T}\bvec{\gamma}\cdot (\GRAD\widetilde{v}_T-\tvec{\eta}_T)\\
={}&\int_T \bvec{\gamma}\cdot (\GRAD\widetilde{v}_T-\tvec{\eta}_T)+\underbrace{\int_T (\vlproj{0}{T}\bvec{\gamma}-\bvec{\gamma})\cdot (\GRAD\widetilde{v}_T-\tvec{\eta}_T)}_{\eqcolon\term_{T,a}}\\
={}&\int_T\bvec{\gamma}\cdot\GRAD \widetilde{v}_T-\int_T \bvec{\gamma}\cdot\PT[0]\uvec{\eta}_T
+\underbrace{\int_T \bvec{\gamma}\cdot(\PT[0]\uvec{\eta}_T-\tvec{\eta}_T)}_{\eqcolon\term_{T,b}}+\term_{T,a}.
\end{align*}
Summing over $T\in\Th$, using the fact that $\widetilde{v}_h\in H^1_0(\Omega)$ (see Lemma \ref{lem:lift.U}) to perform an integration by parts, recalling that $f=-\DIV\bvec{\gamma}$ by \eqref{eq:rm.balance2}, and setting $\term_{\star}\coloneq\sum_{T\in\Th}\term_{T,\star}$ for $\star\in\{a,b\}$, we infer
\begin{align*}
\sum_{T\in\Th}\int_T\bvec{\gamma}\cdot(\GT[0]\underline{v}_T - \PT[0]\uvec{\eta}_T)
&=\int_\Omega f~\widetilde{v}_h-\sum_{T\in\Th}\int_T \bvec{\gamma}\cdot\PT[0]\uvec{\eta}_T+\term_a+\term_b\\
&=\int_\Omega f~\PUh[1]\underline{v}_h
-\sum_{T\in\Th}\int_T \bvec{\gamma}\cdot\PT[0]\uvec{\eta}_T+
\underbrace{\int_\Omega f(\widetilde{v}_h-\PUh[1]\underline{v}_h)}_{\eqcolon\term_c}{}+\term_a+\term_b.
\end{align*}
We plug this relation into the expression \eqref{eq:bh:interpolate} of $\mathrm{b}_h((\Irot{h}\bvec{\theta},\Idis{h}u),(\uvec{\eta}_h,\underline{v}_h))$ and recall \eqref{eq:ah.bis} and the definitions \eqref{eq:Ah} of $\mathrm{A}_h$ and $\ell_h$ to re-write  consistency error \eqref{eq:consistency.error} as
\begin{equation}\label{eq:Eh.decomposition.k0}
  \mathcal{E}_h((\bvec{\theta},u);(\uvec{\eta}_h,\underline{v}_h))
      = \mathcal{E}_{\GRADs,h}(\tens{C}\GRADs\bvec{\theta};\uvec{\eta}_h)
        -\term_a-\term_b-\term_c+\term.
\end{equation}
We now estimate $\term_a$, $\term_b$ and $\term_c$. Using the approximation properties of $\vlproj{0}{T}$ together with Cauchy--Schwarz inequalities and \eqref{lem:lift.Theta.norm}, we have
\begin{equation}\label{est:term_a}
|\term_a|\lesssim h\seminorm[H^1(\Th)^2]{\bvec{\gamma}}\norm[\bvec{\Theta},h]{\uGh[0]\underline{v}_h-\uvec{\eta}_h}
\lesssim h\seminorm[H^1(\Th)^2]{\bvec{\gamma}}\kappa^{-\frac12}t\norm[\bvec{\Theta}\times U,h]{(\uvec{\eta}_h,\underline{v}_h)},
\end{equation}
where we have used the definition \eqref{eq:norm.en.h} of $\norm[\bvec{\Theta}\times U,h]{{\cdot}}$ to conclude.
For $\term_b$, we use again Cauchy--Schwarz inequalities and the estimate \eqref{lem:lift.Theta.diff}, together with the definition of the  norm on $\uvec{\Theta}_{h,0}^0\times\underline{U}_{h,0}^0$, to write
\begin{equation}\label{est:term_b}
|\term_b|\lesssim \norm[L^2(\Omega)^2]{\bvec{\gamma}} h \beta_0^{-\frac12}\norm[\bvec{\Theta}\times U,h]{(\uvec{\eta}_h,\underline{v}_h)}.
\end{equation}
Finally, for $\term_c$, Cauchy--Schwarz inequalities followed by the estimate \eqref{lem:lift.U.diff} yield
\begin{equation}\label{est:term_c}
|\term_c|\lesssim \norm[L^2(\Omega)]{f}h\norm[\bvec{\Theta},h]{\uGh[0]\underline{v}_h}\lesssim
 \norm[L^2(\Omega)]{f}h\mu^{-\frac12}\norm[\bvec{\Theta}\times U,h]{(\uvec{\eta}_h,\underline{v}_h)}.
\end{equation}
Plugging \eqref{est:term_a}--\eqref{est:term_c} and the estimates \eqref{est:T1} on $\term$ and \eqref{est:Egrads} on $\mathcal{E}_{\GRADs,h}(\tens{C}\GRADs\bvec{\theta};\uvec{\eta}_h)$ into \eqref{eq:Eh.decomposition.k0}, we infer
\begin{multline*}
|\mathcal{E}_h((\bvec{\theta},u);(\uvec{\eta}_h,\underline{v}_h))|\\
\lesssim
h\left(\kappa^{-\frac12}t\seminorm[H^1(\Th)^2]{\bvec{\gamma}}+\beta_0^{-\frac12}\norm[L^2(\Omega)^2]{\bvec{\gamma}}
+\mu^{-\frac12}\norm[L^2(\Omega)]{f}
+\beta_0^{-\frac12}(\beta_0 + \beta_1)\seminorm[H^{2}(\Th)^2]{\bvec{\theta}}
\right)
\norm[\bvec{\Theta}\times U,h]{(\uvec{\eta}_h,\underline{v}_h)}.
\end{multline*}
Plugging this estimate into \eqref{eq:basic.err.est} concludes the proof.
\end{proof}

To conclude this section, we provide the proofs of the properties of the liftings.

\begin{proof}[Proof of Lemma \ref{lem:lift.U}]
  1. \emph{Proof of \eqref{lem:lift.U.proj}.}
  Let $\underline{z}_T\in\underline{U}_T^0$. For all $\bvec{\eta}\in \vPoly{0}(T)$, an integration by parts yields
  \[
  \int_T \GRAD\widetilde{z}_T\cdot\bvec{\eta}
  =\sum_{E\in\ET}\int_E \widetilde{z}_T (\bvec{\eta}\cdot\normal_{TE})
  =\sum_{E\in\ET}\int_E z_{\ET}(\bvec{\eta}\cdot\normal_{TE})
  =\int_T\GT[0]\underline{z}_T\cdot\bvec{\eta},
  \]
  where the second equality comes from the definition of $\widetilde{z}_T$ which ensures that $(\widetilde{z}_T)_{|E}=(z_{\ET})_{|E}$ for all $E\in\ET$ (both functions are linear on $E$ and match at the edge's vertices), and the last equality is obtained applying the definition \eqref{eq:def.GT} of $\GT[0]$. This proves \eqref{lem:lift.U.proj}.
  \medskip\\
  2. \emph{Proof of \eqref{lem:lift.U.norm}.}
  For any two vertices $V,V^\star$ of $\ST$ we have by construction
\[
\widetilde{z}_T(\bvec{x}_V)-\widetilde{z}_T(\bvec{x}_{V^\star})=\sum_{W,Z\in\VT}\lambda_{V,W}\lambda_{V^\star,Z}\left(
z_{\ET}(\bvec{x}_{W})-z_{\ET}(\bvec{x}_{Z})
\right).
\]
Integrating the derivative (oriented by each tangent $\tangent_E$) of $z_{\ET}$ on $\partial T$ between $W$ and $Z$ and using the two-dimensional version of the equivalence stated in \cite[Eq.~(4.24)]{Di-Pietro.Droniou:21} between $\norm[\bvec{\Theta},T]{{\cdot}}$ and the component $L^2$-norm, we have
\[
|z_{\ET}(\bvec{x}_{W})-z_{\ET}(\bvec{x}_{Z})|\le \norm[L^1(\partial T)]{z_{\ET}'}\le |\partial T|^{\frac12}\,\norm[L^2(\partial T)]{z_{\ET}'}
\lesssim \norm[\bvec{\Theta},T]{\uGT[0]\underline{z}_T}.
\]
Recalling that $\card(\VT)$ is uniformly bounded by the mesh regularity parameter, and that $(\lambda_{V,W})_{W\in\VT}$ and $(\lambda_{V^\star,Z})_{Z\in\VT}$ are coefficients of convex combinations, we infer from the above relations that
\[
|\widetilde{z}_T(\bvec{x}_V)-\widetilde{z}_T(\bvec{x}_{V^\star})|\lesssim \norm[\bvec{\Theta},T]{\uGT[0]\underline{z}_T}.
\]
Since any edge $e$ of $\ST$ has a length comparable to $h_T$, this shows that $|\GRAD \widetilde{z}_T\cdot\tangent_e|\lesssim h_T^{-1} \norm[\bvec{\Theta},T]{\uGT[0]\underline{z}_T}$ where $\tangent_e$ is any unit tangent to $e$.
Hence, on any triangle $\tau\in\ST$, 
\[
\norm[L^2(\tau)^2]{\GRAD \widetilde{z}_T}
= |\tau|^{\frac12}|(\GRAD\widetilde{z}_T)_{|\tau}|
\le |T|^{\frac12}|(\GRAD\widetilde{z}_T)_{|\tau}|
\lesssim |T|^{\frac12}h_T^{-1} \norm[\bvec{\Theta},T]{\uGT[0]\underline{z}_T}.
\]
Using $|T|^{\frac12}\lesssim h_T$, squaring, summing over $\tau\in \ST$ and taking the square root concludes the proof of \eqref{lem:lift.U.norm}.
\medskip\\
3. \emph{Proof of \eqref{lem:lift.U.diff}.}
We start from the following Poincar\'e inequality with trace:
\begin{equation}\label{eq:poincare.trace}
\norm[L^2(T)]{w_T}^2\lesssim h_T^2\norm[L^2(T)^2]{\GRAD w_T}^2+\sum_{E\in\ET}h_E\norm[L^2(E)]{w_T}^2\qquad\forall w_T\in \Poly[\rm c]{1}(\ST).
\end{equation}
To prove this estimate, consider a triangle $\tau\in \ST$ with an edge $e\subset \partial T$. Taking $\bvec{x}\in \tau$ and $\bvec{y}\in e$ we have $|w_T(\bvec{x})|\le h_T|(\GRAD w_T)_{|\tau}|+|w_T(\bvec{y})|$; integrating over $\bvec{y}\in e$, squaring, integrating over $\bvec{x}\in \tau$ and using $|\tau|/h_e\lesssim h_e$ (by shape regularity) leads to 
\begin{equation}\label{eq:poincare.tau}
\norm[L^2(\tau)]{w_T}^2\lesssim h_T^2\norm[L^2(\tau)^2]{\GRAD w_T}^2+h_e\norm[L^2(e)]{w_T}^2.
\end{equation}
If all triangles in $\ST$ have an edge $e\subset \partial T$, summing \eqref{eq:poincare.tau} over $\tau\in \ST$ concludes the proof of \eqref{eq:poincare.trace}; otherwise, a discrete trace inequality and \eqref{eq:poincare.tau} give a bound on the trace of $w_T$ on the other edges of $\tau$, and the process can be iterated on the triangles in $\ST$ that touch $\tau$ but do not have an edge on $\partial T$.

Applying \eqref{eq:poincare.trace} to $w_T=\widetilde{z}_T-\PUT[1]\underline{z}_T$, using a triangle inequality, $h_E\le h_T$, the estimate \eqref{lem:lift.U.norm}, and the fact that $(\widetilde{z}_T)_{|E}=(z_{\ET})_{|E}$ for all $E\in\ET$, we obtain
\[
\norm[L^2(T)]{\widetilde{z}_T-\PUT[1]\underline{z}_T}^2\lesssim h_T^2\left(\norm[\bvec{\Theta},T]{\uGT[0]\underline{z}_T}^2+\norm[L^2(T)^2]{\GRAD \PUT[1]\underline{z}_T}^2\right)+h_T^2\!\sum_{E\in\ET}\!h_E^{-1}\norm[L^2(E)]{z_{\ET}-\PUT[1]\underline{z}_T}^2.
\]
The proof of \eqref{lem:lift.U.diff} is completed by invoking \cite[Lemma 35 and Eq.~(4.24)]{Di-Pietro.Droniou:21} to write
\[
\norm[L^2(T)^2]{\GRAD \PUT[1]\underline{z}_T}^2+\sum_{E\in\ET}h_E^{-1}\norm[L^2(E)]{z_{\ET}-\PUT[1]\underline{z}_T}^2\lesssim
\norm[\bvec{\Theta},T]{\uGT[0]\underline{z}_T}^2.
\qedhere
\]
\end{proof}

\begin{proof}[Proof of Lemma \ref{lem:lift.Theta}]
  1. \emph{Proof of \eqref{lem:lift.Theta.proj}.}
  By \eqref{lem:lift.U.proj}, 
  $
  \vlproj{0}{T}\tvec{\eta}_T
  = \GT[0]\underline{w}_T+\PT[0]\uvec{\kappa}_T
  = \PT[0](\uGT[0]\underline{w}_T+\uvec{\kappa}_T),
  $
  where the last equality follows from the relation $\PT[0]\uGT[0]=\GT[0]$, see \cite[Eq.~(3.22)]{Di-Pietro.Droniou:21}.
  This proves that $\vlproj{0}{T}\tvec{\eta}_T=\PT[0]\uvec{\eta}_T^\flat$. Since $\PT[0]$ depends only on the tangential components of $\uvec{\eta}_T$ (see \eqref{eq:def.RT} and \eqref{eq:def.PT}), this concludes the proof of \eqref{lem:lift.Theta.proj}.
  \medskip\\
  2. \emph{Proof of \eqref{lem:lift.Theta.norm}.}
  We use the definition of $\tvec{\eta}_T$ to write $\GRAD\widetilde{v}_T-\tvec{\eta}_T=\GRAD(\widetilde{v}_T-\widetilde{w}_T)-\PT[0]\uvec{\kappa}_T$ and thus, by triangle inequality,
  \begin{align*}
    \norm[L^2(T)^2]{\GRAD\widetilde{v}_T-\tvec{\eta}_T}^2\le{}&2\norm[L^2(T)^2]{\GRAD(\widetilde{v}_T-\widetilde{w}_T)}^2+2\norm[L^2(T)^2]{\PT[0]\uvec{\kappa}_T}^2\\
    \lesssim{}& \norm[\bvec{\Theta},T]{\uGT[0](\underline{v}_T-\underline{w}_T)}^2+\norm[\bvec{\Theta},T]{\uvec{\kappa}_T}^2\\
    ={}&\norm[\bvec{\Theta},T]{\uGT[0](\underline{v}_T-\underline{w}_T)-\uvec{\kappa}_T}^2,
  \end{align*}
  where the second line follows from \eqref{lem:lift.U.norm} applied to $\underline{z}_T=\underline{v}_T-\underline{w}_T$ and the estimate $\norm[L^2(T)^2]{\PT[0]\uvec{\kappa}_T}\lesssim \norm[\bvec{\Theta},T]{\uvec{\kappa}_T}$ (see \cite[Proposition 27]{Di-Pietro.Droniou:21}), while the conclusion is obtained using the orthogonality for the $(\cdot,\cdot)_{\bvec{\Theta},T}$ product of $\uvec{\kappa}_T$ and $\uGT[0](\underline{v}_T-\underline{w}_T)$. This gives
  \[
  \norm[L^2(T)^2]{\GRAD\widetilde{v}_T-\tvec{\eta}_T}\lesssim \norm[\bvec{\Theta},T]{\uGT[0]\underline{v}_T-\uvec{\eta}_T^\flat}
  \]
  and the proof of \eqref{lem:lift.Theta.norm} is complete since $\norm[\bvec{\Theta},T]{{\cdot}}$ depends only on tangential components of vectors in $\uvec{\Theta}_T^0$.
  \medskip\\
  3. \emph{Proof of \eqref{lem:lift.Theta.diff}.}
  Let $\phi_T(\bvec{x})\coloneq\PT[0]\uvec{\eta}_T\cdot(\bvec{x}-\bvec{x}_T)\in\Poly{1}(T)$.
  By \eqref{eq:tilde.linearly.exact}, we have $\PT[0]\uvec{\eta}_T=\GRAD\phi_T=\GRAD\widetilde{\Idisz{T}\phi_T}$, and thus \eqref{lem:lift.Theta.norm} with $\underline{v}_T=\Idisz{T}\phi_T$ yields
  \begin{align*}
    \norm[L^2(T)^2]{\PT[0]\uvec{\eta}_T-\tvec{\eta}_T}^2
    &\lesssim \norm[\bvec{\Theta},T]{\uGT[0](\Idisz{T}\phi_T)-\uvec{\eta}_T}^2
    \\
    &\lesssim \sum_{E\in\ET}h_E\norm[L^2(E)]{\PT[0]\uvec{\eta}_T\cdot\tangent_E-\bvec{\eta}_E\cdot\tangent_E}^2
    \lesssim h_T^2\sum_{E\in\ET}h_E^{-1}\norm[L^2(E)]{\PT[0]\uvec{\eta}_T-\bvec{\eta}_E}^2,
  \end{align*}
  where the second inequality follows from the two-dimensional version of the norm equivalence \cite[Eq.~(4.24)]{Di-Pietro.Droniou:21} together with
  the local version of \eqref{eq:tIrot:P2} which gives $\uGT[0](\Idisz{T}\phi_T)=\tIrotz{T}(\GRAD \phi_T)=\tIrotz{T}(\PT[0]\uvec{\eta}_T)$.
  The estimate \eqref{lem:lift.Theta.diff} follows writing $h_T\le h$, summing over $T\in\Th$, and invoking \cite[Eq.~(7.103) and (7.109)]{Di-Pietro.Droniou:20} to see that
  \[
  \sum_{T\in\Th}\sum_{E\in\ET}h_E^{-1}\norm[L^2(E)^2]{\PT[0]\uvec{\eta}_T-\bvec{\eta}_E}^2
  \lesssim\norm[L^2(\Omega)^{2\times 2}]{\Gsh[0]\uvec{\eta}_h}^2+\seminorm[{\rm s,j},h]{\uvec{\eta}_h}^2.
  \qedhere\]
\end{proof}


\section{Numerical results}\label{sec:numerical.results}

We illustrate the practical behaviour of the DDR scheme \eqref{eq:discrete} through three different numerical tests. In Sections \ref{sec:test.polynomial} and \ref{sec:test.analytical}, we assess the convergence towards two different analytical solutions of the Reissner--Mindlin problem with clamped boundary conditions and three families of meshes of $\Omega=(0,1)^2$: (mostly) hexagonal meshes, triangular meshes, and locally refined meshes (with hanging nodes); Figure \ref{fig:meshes} shows a representative member of each family of meshes.
We focus on the $h$-convergence for polynomial degrees $0\le k\le 3$, check the convergence rates, and discuss the robustness of the scheme with respect to the thickness $t$ of the plate. 

In Section \ref{sec:tests.kirchoff}, we consider the limiting case of the Kirchoff--Love model ($t\to 0$ in \eqref{eq:rm.strong}) and assess how the lowest-order DDR scheme \eqref{eq:discrete} behaves when using simply supported boundary conditions; we also compare it with a stabilised finite element method.

The DDR tools and the scheme have been implemented in the \texttt{HArDCore2D} C++ framework (see \url{https://github.com/jdroniou/HArDCore}), which is based on linear algebra facilities from the \texttt{Eigen3} library (see \url{http://eigen.tuxfamily.org}).
The resolution of the global sparse linear systems uses the direct solver from the \texttt{Intel MKL PARDISO} library (see \url{https://software.intel.com/en-us/mkl}).

In all the tests, the Young modulus is taken as $E=1$, while the Poisson ratio is $\nu=0.3$. The energy error is computed as the (relative) $\norm[\bvec{\Theta}\times U,h]{{\cdot}}$-norm of the difference between the approximate solution and the interpolate of the exact solution, that is:
\begin{equation}\label{eq:def.error}
E_h\coloneq \frac{\norm[\bvec{\Theta}\times U,h]{(\uvec{\theta}_h-\Irot{h}\bvec{\theta},\underline{u}_h-\Idis{h}u)}}{\norm[\bvec{\Theta}\times U,h]{(\Irot{h}\bvec{\theta},\Idis{h}u)}}.
\end{equation}

\begin{figure}\centering
  \begin{minipage}{0.3\textwidth}\centering
    \includegraphics[height=4cm]{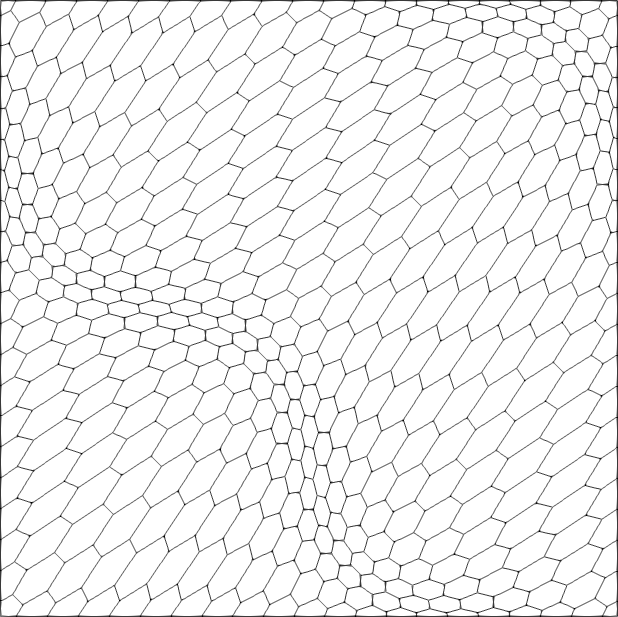}
    \subcaption{Hexagonal mesh\label{fig:hexa}}
  \end{minipage}
  \hspace{0.25cm}
  \begin{minipage}{0.3\textwidth}\centering
    \includegraphics[height=4cm]{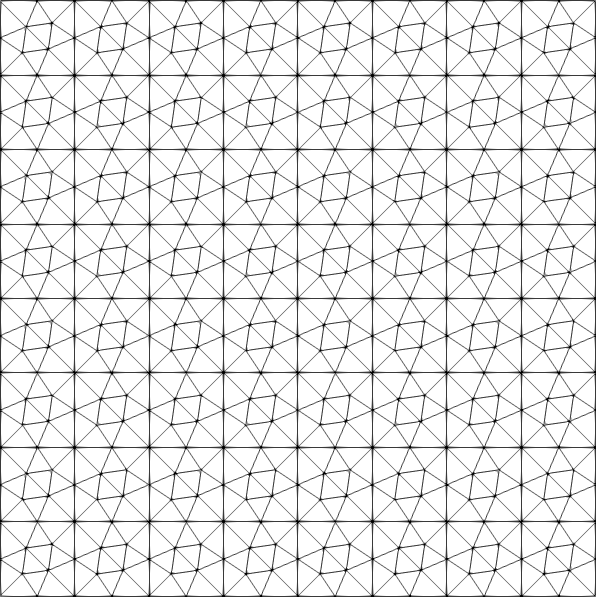}
    \subcaption{Triangular mesh}
  \end{minipage}
  \vspace{0.25cm}
  \begin{minipage}{0.3\textwidth}\centering
    \includegraphics[height=4cm]{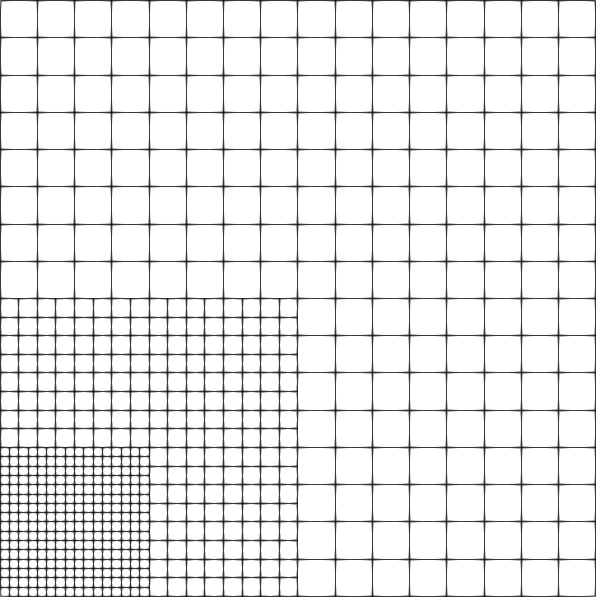}
    \subcaption{Locally refined mesh}
  \end{minipage}
  \caption{Members of mesh families used in numerical tests.\label{fig:meshes}}
\end{figure}


\subsection{Polynomial solution}\label{sec:test.polynomial}

The first series of tests is run with source term corresponding to the following exact polynomial solution introduced in \cite{Chinosi.Lovadina.ea:06}:
\begin{align*}
u(\bvec{x})={}&\frac13 x_1^3(1-x_1^3)x_2^3(1-x_2)^3\\
&-\frac{2t^2}{5(1-\nu)}\left[x_2^3(x_2-1)^3x_1(x_1-1)(5x_1^2-5x_1+1)+x_1^3(x_1-1)^3x_2(x_2-1)(5x_2^2-5x_2+1)\right],\\
\bvec{\theta}(\bvec{x})={}&\begin{bmatrix}
  x_2^3(x_2-1)^3x_1^2(x_1-1)^2(2x_1-1) \\[2pt]
  x_1^3(x_1-1)^3x_2^2(x_2-1)^2(3x_2-1)
\end{bmatrix}.
\end{align*}

The results are presented in Figure \ref{fig:conv.polynomial}. We notice that, for all considered polynomial degrees $k\in \{0,1,2,3\}$ and thicknesses $t\in\{10^{-1},10^{-3}\}$, the error decays as $h^{k+1}$ (as expected from Theorem \ref{th:error.est}) and is mostly independent of $t$. The same observation can be made for $k\in\{0,1\}$ and $t=10^{-5}$. However, for $k\ge 2$, we notice an apparent loss of convergence on the finest meshes when $t=10^{-5}$.
This phenomenon is not a manifestation of lack of robustness of the scheme as, for the considered solution, the $H^s$-norms of the variables (displacement, rotation, shear strain) remain uniformly bounded with respect to $t$, and Theorem \ref{th:error.est} thus shows that we should expect a convergence in $\mathcal O(h^{k+1})$ with multiplicative constants that are independent of $t$.
This apparent loss of convergence occurs when the condition number of the system matrix becomes very large ($10^{17}$ or more), and corresponds to extreme values of the thickness that are not encountered in practical applications (the value $t=10^{-5}$ corresponds to a plate with characteristic height/length of $1\,{\rm m}$ and thickness of $10\,\mu{\rm m}$).
In these situations, despite the use of a direct solver, the residual becomes large ($\simeq 10^{-5}$) owing to the propagation of round-off errors, leading to a poorly resolved linear system.
A similar phenomenon had been noticed in \cite{Holzer.Rank.ea:90}.

We notice that the tests we present here are among the few on high-order schemes for the Reissner--Mindlin plate model. 
In \cite{Beirao-da-Veiga.Hughes.ea:15}, isogeometric schemes are considered up to a polynomial degree 5, corres\-ponding to a convergence rate in $h^4$, and thus to the choice $k=3$ in the DDR scheme. The smallest thickness considered in this reference is $t=10^{-3}$, and the
largest mesh has about 300 rectangles; at these levels, no rounding error is noticeable in our tests (to compare, the second locally refined mesh we consider has more than 600 elements, and the largest one has more than 10,000 elements).
An over-penalised discontinuous Galerkin scheme is presented in \cite{Bosing.Carstensen:15}, and tests are produced with a
polynomial degree 4 (convergence rate in $h^3$), corresponding to $k=2$ for the DDR scheme. Very thin plates are considered in these
tests, with $t$ as small as $10^{-6}$; however, for this thickness, the largest triangular mesh in the tests of \cite{Bosing.Carstensen:15} has 512 triangles; our second coarsest triangular
mesh has 896 triangles and, as can be seen in Figure \ref{fig:poly.tri}, for $k=2$ and at this size of mesh the convergence is not affected by the increase in the condition number of the matrix.
It therefore seems that those previous tests were carried out under conditions in which the conditioning of the matrix remains reasonable, and that the tests we present here are the first ones to highlight this phenomenon for high-order schemes and very thin plates.

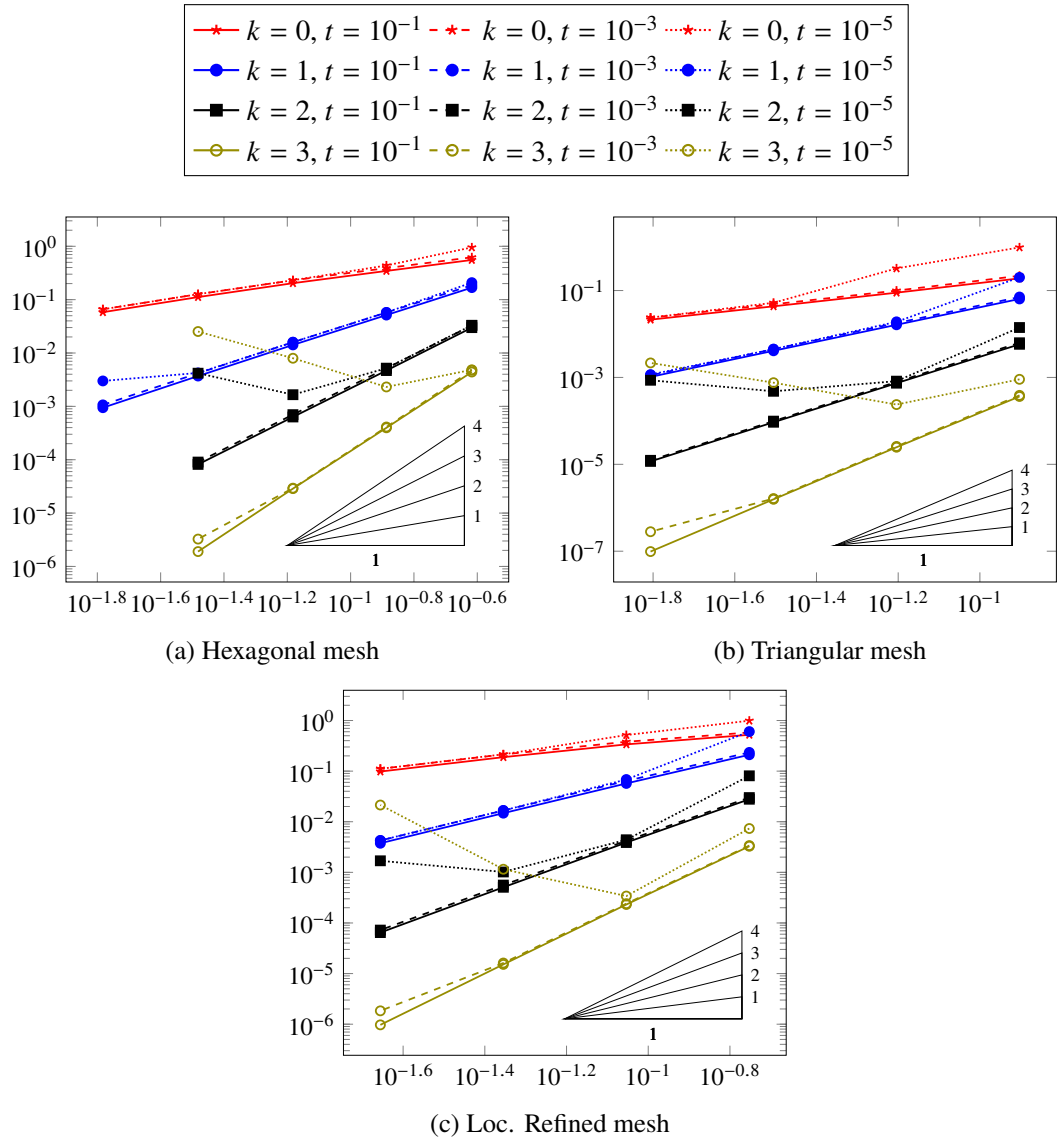
\begin{figure}\centering
  \ref{conv.hexa}
  \vspace{0.50cm}\\
  \begin{minipage}{0.45\textwidth}
    \begin{tikzpicture}[scale=0.85]
      \begin{loglogaxis} 
        \addplot [thick, mark=star, red] table[x=MeshSize,y=Error] {outputs/polynomial_solution/hexa_k0_t1e-1/data_rates.dat};
        \logLogSlopeTriangle{0.90}{0.4}{0.1}{1}{black};
        \addplot [thick, mark=star, mark options=solid, red, dashed] table[x=MeshSize,y=Error] {outputs/polynomial_solution/hexa_k0_t1e-3/data_rates.dat};
        \addplot [thick, mark=star, mark options=solid, red, densely dotted] table[x=MeshSize,y=Error] {outputs/polynomial_solution/hexa_k0_t1e-5/data_rates.dat};
        \addplot [thick, mark=*, blue] table[x=MeshSize,y=Error] {outputs/polynomial_solution/hexa_k1_t1e-1/data_rates.dat};
        \logLogSlopeTriangle{0.90}{0.4}{0.1}{2}{black};
        \addplot [thick, mark=*, mark options=solid, blue, dashed] table[x=MeshSize,y=Error] {outputs/polynomial_solution/hexa_k1_t1e-3/data_rates.dat};
        \addplot [thick, mark=*, mark options=solid, blue, densely dotted] table[x=MeshSize,y=Error] {outputs/polynomial_solution/hexa_k1_t1e-5/data_rates.dat};
        \addplot [thick, mark=square*, black] table[x=MeshSize,y=Error] {outputs/polynomial_solution/hexa_k2_t1e-1/data_rates.dat};
        \logLogSlopeTriangle{0.90}{0.4}{0.1}{3}{black};
        \addplot [thick, mark=square*, mark options=solid, black, dashed] table[x=MeshSize,y=Error] {outputs/polynomial_solution/hexa_k2_t1e-3/data_rates.dat};
        \addplot [thick, mark=square*, mark options=solid, black, densely dotted] table[x=MeshSize,y=Error] {outputs/polynomial_solution/hexa_k2_t1e-5/data_rates.dat};
        \addplot [thick, mark=o, olive] table[x=MeshSize,y=Error] {outputs/polynomial_solution/hexa_k3_t1e-1/data_rates.dat};
        \logLogSlopeTriangle{0.90}{0.4}{0.1}{4}{black};
        \addplot [thick, mark=o, mark options=solid, olive, dashed] table[x=MeshSize,y=Error] {outputs/polynomial_solution/hexa_k3_t1e-3/data_rates.dat};
        \addplot [thick, mark=o, mark options=solid, olive, densely dotted] table[x=MeshSize,y=Error] {outputs/polynomial_solution/hexa_k3_t1e-5/data_rates.dat};
      \end{loglogaxis}            
    \end{tikzpicture}
    \subcaption{Hexagonal mesh}
  \end{minipage}
  \begin{minipage}{0.45\textwidth}
    \begin{tikzpicture}[scale=0.85]
      \begin{loglogaxis}
        \addplot [thick, mark=star, red] table[x=MeshSize,y=Error] {outputs/polynomial_solution/tri_k0_t1e-1/data_rates.dat};
        \logLogSlopeTriangle{0.90}{0.4}{0.1}{1}{black};
        \addplot [thick, mark=star, mark options=solid, red, dashed] table[x=MeshSize,y=Error] {outputs/polynomial_solution/tri_k0_t1e-3/data_rates.dat};
        \addplot [thick, mark=star, mark options=solid, red, densely dotted] table[x=MeshSize,y=Error] {outputs/polynomial_solution/tri_k0_t1e-5/data_rates.dat};
        \addplot [thick, mark=*, blue] table[x=MeshSize,y=Error] {outputs/polynomial_solution/tri_k1_t1e-1/data_rates.dat};
        \logLogSlopeTriangle{0.90}{0.4}{0.1}{2}{black};
        \addplot [thick, mark=*, mark options=solid, blue, dashed] table[x=MeshSize,y=Error] {outputs/polynomial_solution/tri_k1_t1e-3/data_rates.dat};
        \addplot [thick, mark=*, mark options=solid, blue, densely dotted] table[x=MeshSize,y=Error] {outputs/polynomial_solution/tri_k1_t1e-5/data_rates.dat};
        \addplot [thick, mark=square*, black] table[x=MeshSize,y=Error] {outputs/polynomial_solution/tri_k2_t1e-1/data_rates.dat};
        \logLogSlopeTriangle{0.90}{0.4}{0.1}{3}{black};
        \addplot [thick, mark=square*, mark options=solid, black, dashed] table[x=MeshSize,y=Error] {outputs/polynomial_solution/tri_k2_t1e-3/data_rates.dat};
        \addplot [thick, mark=square*, mark options=solid, black, densely dotted] table[x=MeshSize,y=Error] {outputs/polynomial_solution/tri_k2_t1e-5/data_rates.dat};
        \addplot [thick, mark=o, olive] table[x=MeshSize,y=Error] {outputs/polynomial_solution/tri_k3_t1e-1/data_rates.dat};
        \logLogSlopeTriangle{0.90}{0.4}{0.1}{4}{black};
        \addplot [thick, mark=o, mark options=solid, olive, dashed] table[x=MeshSize,y=Error] {outputs/polynomial_solution/tri_k3_t1e-3/data_rates.dat};
        \addplot [thick, mark=o, mark options=solid, olive, densely dotted] table[x=MeshSize,y=Error] {outputs/polynomial_solution/tri_k3_t1e-5/data_rates.dat};
      \end{loglogaxis}            
    \end{tikzpicture}
    \subcaption{Triangular mesh}\label{fig:poly.tri}
  \end{minipage}
  \vspace{0.25cm}\\
  \begin{minipage}{0.45\textwidth}
    \begin{tikzpicture}[scale=0.85]
      \begin{loglogaxis}[legend columns=3, legend to name=conv.hexa]       
        \addplot [thick, mark=star, red] table[x=MeshSize,y=Error] {outputs/polynomial_solution/locref_k0_t1e-1/data_rates.dat};
        \addlegendentry{$k=0$, $t=10^{-1}$}
        \logLogSlopeTriangle{0.90}{0.4}{0.1}{1}{black};
        \addplot [thick, mark=star, mark options=solid, red, dashed] table[x=MeshSize,y=Error] {outputs/polynomial_solution/locref_k0_t1e-3/data_rates.dat};
        \addlegendentry{$k=0$, $t=10^{-3}$}
        \addplot [thick, mark=star, mark options=solid, red, densely dotted] table[x=MeshSize,y=Error] {outputs/polynomial_solution/locref_k0_t1e-5/data_rates.dat};
        \addlegendentry{$k=0$, $t=10^{-5}$}
        \addplot [thick, mark=*, blue] table[x=MeshSize,y=Error] {outputs/polynomial_solution/locref_k1_t1e-1/data_rates.dat};
        \addlegendentry{$k=1$, $t=10^{-1}$}
        \logLogSlopeTriangle{0.90}{0.4}{0.1}{2}{black};
        \addplot [thick, mark=*, mark options=solid, blue, dashed] table[x=MeshSize,y=Error] {outputs/polynomial_solution/locref_k1_t1e-3/data_rates.dat};
        \addlegendentry{$k=1$, $t=10^{-3}$}
        \addplot [thick, mark=*, mark options=solid, blue, densely dotted] table[x=MeshSize,y=Error] {outputs/polynomial_solution/locref_k1_t1e-5/data_rates.dat};
        \addlegendentry{$k=1$, $t=10^{-5}$}
        \addplot [thick, mark=square*, black] table[x=MeshSize,y=Error] {outputs/polynomial_solution/locref_k2_t1e-1/data_rates.dat};
        \addlegendentry{$k=2$, $t=10^{-1}$}
        \logLogSlopeTriangle{0.90}{0.4}{0.1}{3}{black};
        \addplot [thick, mark=square*, mark options=solid, black, dashed] table[x=MeshSize,y=Error] {outputs/polynomial_solution/locref_k2_t1e-3/data_rates.dat};
        \addlegendentry{$k=2$, $t=10^{-3}$}
        \addplot [thick, mark=square*, mark options=solid, black, densely dotted] table[x=MeshSize,y=Error] {outputs/polynomial_solution/locref_k2_t1e-5/data_rates.dat};
        \addlegendentry{$k=2$, $t=10^{-5}$}
        \addplot [thick, mark=o, olive] table[x=MeshSize,y=Error] {outputs/polynomial_solution/locref_k3_t1e-1/data_rates.dat};
        \logLogSlopeTriangle{0.90}{0.4}{0.1}{4}{black};
        \addlegendentry{$k=3$, $t=10^{-1}$}
        \addplot [thick, mark=o, mark options=solid, olive, dashed] table[x=MeshSize,y=Error] {outputs/polynomial_solution/locref_k3_t1e-3/data_rates.dat};
        \addlegendentry{$k=3$, $t=10^{-3}$}
        \addplot [thick, mark=o, mark options=solid, olive, densely dotted] table[x=MeshSize,y=Error] {outputs/polynomial_solution/locref_k3_t1e-5/data_rates.dat};
        \addlegendentry{$k=3$, $t=10^{-5}$}
      \end{loglogaxis}            
    \end{tikzpicture}
    \subcaption{Loc. Refined mesh}
  \end{minipage}
  \caption{Error $E_h$ (see \eqref{eq:def.error}) with respect to the meshsize $h$ for the polynomial solution of Section \ref{sec:test.polynomial}.\label{fig:conv.polynomial}}
\end{figure}

\subsection{Analytical solution with improved physical behaviour}\label{sec:test.analytical}

As explained in \cite[Theorem 2.1 and following remarks]{Arnold.Falk:89}, as $t\to 0$ the shear strain $\bvec{\gamma}$ is expected to remain bounded in $L^2$-norm, but to grow unboundedly in $H^1$-norm. The polynomial solution considered in Section \ref{sec:test.polynomial} does not reproduce this behaviour (for this solution, the shear strain is actually independent of $t$). To test our DDR scheme in a setting which is at least quantitatively closer to the generic physical behaviour of the Reissner--Mindlin model, we design in this section a new analytical solution on $\Omega=(0,1)^2$ (with non-homogeneous boundary conditions), with the following behaviour as $t\to 0$:
\begin{equation}\label{eq:analytical.behaviour}
  \begin{gathered}
    \text{
      $\norm[H^3(\Omega)]{u}\sim 1$,
      $\norm[H^2(\Omega)^2]{\bvec{\theta}}\sim 1$,
      $\norm[L^2(\Omega)^2]{\bvec{\gamma}}\sim 1$,
      $\seminorm[H^s(\Omega)^2]{\bvec{\gamma}}\sim t^{-s+\frac12}$
      for all $s\ge 1$},\\
    \text{and $f$ is independent of $t$}.
  \end{gathered}
\end{equation}
As noticed in \cite[Theorem 2.1]{Arnold.Falk:89}, the expected growth of $\seminorm[H^1(\Omega)]{\bvec{\gamma}}$ is in $t^{-1}$, not $t^{-\frac12}$ as in the solution we construct here. This solution could easily be adjusted to produce such a growth, but this would come at the cost of extremely steep dependency on $t$ (in particular, a term $t^6$ in \eqref{eq:def.u.theta} below) that would make the resolution even more challenging. 

\subsubsection{Design of the solution}

We look for a solution under the form
\begin{equation}\label{eq:def.u.theta}
  \begin{gathered}
    \text{
      $u(\bvec{x})=v(t,\bvec{x})+t^2w(t,\bvec{x})$
      and
      $\bvec{\theta}(\bvec{x})=\GRAD v(t,\bvec{x})$,
    }\\
    \text{where $v(t,\bvec{x})=t^3 V(t^{-1}\bvec{x}) + g(\bvec{x})$
      with $V(\bvec{y})=y_1e^{-y_1}\cos(y_2)$
      and $g(\bvec{x})=\sin(\pi x_1)\sin(\pi x_2)$.
    }
  \end{gathered}
\end{equation}
Defining $\bvec{\gamma}$ by \eqref{eq:rm.def.gamma} gives $\bvec{\gamma}(t,\cdot)=\kappa \GRAD w(t,\cdot)$. The function $w$ is then selected to ensure that \eqref{eq:rm.balance1} holds.
Since $\vDIV (\tens{C}\GRADs\bvec{\theta})=\vDIV(\tens{C}\GRADs(\GRAD v))=(\beta_0+\beta_1)\GRAD\Delta v$, \eqref{eq:rm.balance1} corresponds to
\[
w(t,\cdot)=-\frac{\beta_0+\beta_1}{\kappa}\Delta v(t,\cdot).
\]
The transverse load $f$ is fixed according to \eqref{eq:rm.balance2}:
\[
f(t,\cdot)=(\beta_0+\beta_1)\Delta^2 v(t,\cdot).
\]
Let us now briefly check that \eqref{eq:analytical.behaviour} holds. We first notice that, for any natural numbers $m,n$, the mapping $(0,\infty)\times \Omega\ni(t,\bvec{x})\mapsto (\partial_1^m\partial_2^nV)(t^{-1}\bvec{x})$ is uniformly bounded. This shows that $\norm[H^3(\Omega)]{u}$, $\norm[H^2(\Omega)^2]{\bvec{\theta}}$ and $\norm[L^2(\Omega)^2]{\bvec{\gamma}}$ remain bounded as $t\to 0$; these norms also do not go to zero owing to the presence of $g$ (which could actually be any smooth function with non-zero derivatives up to order 3). The function $V$ satisfies $\Delta V(\bvec{y})=-2e^{-y_1}\cos(y_2)$ and thus $\Delta^2V=0$; hence, $f=(\beta_0+\beta_1)\Delta^2 g$ is independent of $t$. This also shows that
\[
\bvec{\gamma}(\bvec{x})=-2(\beta_0+\beta_1)e^{-t^{-1}y_1}\left[\begin{array}{c}\cos(t^{-1}y_2)\\ \sin(t^{-1}y_2) \end{array}\right]-(\beta_0+\beta_1)\GRAD\Delta g(\bvec{x}).
\]
For a given $s\ge 1$, taking any partial derivative of order $s$ of this expression and using $\norm[L^2(0,1)]{e^{-t^{-1}\bullet}}\sim t^{1/2}$ shows that $\seminorm[H^s(\Omega)^2]{\bvec{\gamma}}\sim t^{-s+\frac12}$.

\subsubsection{Results}

The results of the numerical tests with the analytical solution \eqref{eq:def.u.theta} are presented in Figure \ref{fig:conv.analytical}. We observe a similar behaviour as in the numerical results for the polynomial solution
(see Figure \ref{fig:conv.polynomial}). The scheme is here completely robust for $k=0$ (as expected from Theorem \ref{th:error.est.k0}), and for $k=1$ up to $t=10^{-3}$ (and also for $t=10^{-5}$ up to errors of magnitude $10^{-3}$). The degradation of convergence occurs however sooner, with respect to increasing $k$ or $1/t$, than in Section \ref{sec:test.polynomial}: the apparent loss of convergence is here already perceptible for $(k,t)=(1,10^{-5})$ or $(k,t)=(3,10^{-3})$ for example; it also seems more severe for $t=10^{-5}$ and $k\ge 1$.

This is not completely unexpected as the dependency of the analytical solution \eqref{eq:def.u.theta} with respect to $t$ is more severe, and the higher-order norms of the shear strain indeed grows with $t$ here. Combined with the round-off errors phenomenon previously mentioned, this explains the worse numerical behaviour. We however notice that the scheme remains more robust, even for higher degrees, than what the error estimate \eqref{eq:err.est} could lead us to believe; considering for example $k=3$, since $\seminorm[H^{4}(\Th)]{\bvec{\gamma}}$ grows as $t^{-3.5}$, the upper bound on the error in \eqref{eq:err.est} grows between $t=10^{-1}$ and $t=10^{-3}$ by a factor $10^{2\times 3.5}=10^7$, which is clearly not the case of the error itself (on the finest mesh, the ratio of the errors for these two values of $t$ is at most $10^3$).

\begin{figure}\centering
  \ref{conv.hexa.analytical}
  \vspace{0.50cm}\\
  \begin{minipage}{0.45\textwidth}
    \begin{tikzpicture}[scale=0.85]
      \begin{loglogaxis}    
        \addplot [thick, mark=star, red] table[x=MeshSize,y=Error] {outputs/analytical_solution/hexa_k0_t1e-1/data_rates.dat};
        \logLogSlopeTriangle{0.90}{0.4}{0.1}{1}{black};
        \addplot [thick, mark=star, mark options=solid, red, dashed] table[x=MeshSize,y=Error] {outputs/analytical_solution/hexa_k0_t1e-3/data_rates.dat};
        \addplot [thick, mark=star, mark options=solid, red, densely dotted] table[x=MeshSize,y=Error] {outputs/analytical_solution/hexa_k0_t1e-5/data_rates.dat};
        \addplot [thick, mark=*, blue] table[x=MeshSize,y=Error] {outputs/analytical_solution/hexa_k1_t1e-1/data_rates.dat};
        \logLogSlopeTriangle{0.90}{0.4}{0.1}{2}{black};
        \addplot [thick, mark=*, mark options=solid, blue, dashed] table[x=MeshSize,y=Error] {outputs/analytical_solution/hexa_k1_t1e-3/data_rates.dat};
        \addplot [thick, mark=*, mark options=solid, blue, densely dotted] table[x=MeshSize,y=Error] {outputs/analytical_solution/hexa_k1_t1e-5/data_rates.dat};
        \addplot [thick, mark=square*, black] table[x=MeshSize,y=Error] {outputs/analytical_solution/hexa_k2_t1e-1/data_rates.dat};
        \logLogSlopeTriangle{0.90}{0.4}{0.1}{3}{black};
        \addplot [thick, mark=square*, mark options=solid, black, dashed] table[x=MeshSize,y=Error] {outputs/analytical_solution/hexa_k2_t1e-3/data_rates.dat};
        \addplot [thick, mark=square*, mark options=solid, black, densely dotted] table[x=MeshSize,y=Error] {outputs/analytical_solution/hexa_k2_t1e-5/data_rates.dat};
        \addplot [thick, mark=o, olive] table[x=MeshSize,y=Error] {outputs/analytical_solution/hexa_k3_t1e-1/data_rates.dat};
        \logLogSlopeTriangle{0.90}{0.4}{0.1}{4}{black};
        \addplot [thick, mark=o, mark options=solid, olive, dashed] table[x=MeshSize,y=Error] {outputs/analytical_solution/hexa_k3_t1e-3/data_rates.dat};
        \addplot [thick, mark=o, mark options=solid, olive, densely dotted] table[x=MeshSize,y=Error] {outputs/analytical_solution/hexa_k3_t1e-5/data_rates.dat};
      \end{loglogaxis}            
    \end{tikzpicture}
    \subcaption{Hexagonal mesh}
  \end{minipage}
  \begin{minipage}{0.45\textwidth}
    \begin{tikzpicture}[scale=0.85]
      \begin{loglogaxis}
        \addplot [thick, mark=star, red] table[x=MeshSize,y=Error] {outputs/analytical_solution/tri_k0_t1e-1/data_rates.dat};
        \logLogSlopeTriangle{0.90}{0.4}{0.1}{1}{black};
        \addplot [thick, mark=star, mark options=solid, red, dashed] table[x=MeshSize,y=Error] {outputs/analytical_solution/tri_k0_t1e-3/data_rates.dat};
        \addplot [thick, mark=star, mark options=solid, red, densely dotted] table[x=MeshSize,y=Error] {outputs/analytical_solution/tri_k0_t1e-5/data_rates.dat};
        \addplot [thick, mark=*, blue] table[x=MeshSize,y=Error] {outputs/analytical_solution/tri_k1_t1e-1/data_rates.dat};
        \logLogSlopeTriangle{0.90}{0.4}{0.1}{2}{black};
        \addplot [thick, mark=*, mark options=solid, blue, dashed] table[x=MeshSize,y=Error] {outputs/analytical_solution/tri_k1_t1e-3/data_rates.dat};
        \addplot [thick, mark=*, mark options=solid, blue, densely dotted] table[x=MeshSize,y=Error] {outputs/analytical_solution/tri_k1_t1e-5/data_rates.dat};
        \addplot [thick, mark=square*, black] table[x=MeshSize,y=Error] {outputs/analytical_solution/tri_k2_t1e-1/data_rates.dat};
        \logLogSlopeTriangle{0.90}{0.4}{0.1}{3}{black};
        \addplot [thick, mark=square*, mark options=solid, black, dashed] table[x=MeshSize,y=Error] {outputs/analytical_solution/tri_k2_t1e-3/data_rates.dat};
        \addplot [thick, mark=square*, mark options=solid, black, densely dotted] table[x=MeshSize,y=Error] {outputs/analytical_solution/tri_k2_t1e-5/data_rates.dat};
        \addplot [thick, mark=o, olive] table[x=MeshSize,y=Error] {outputs/analytical_solution/tri_k3_t1e-1/data_rates.dat};
        \logLogSlopeTriangle{0.90}{0.4}{0.1}{4}{black};
        \addplot [thick, mark=o, mark options=solid, olive, dashed] table[x=MeshSize,y=Error] {outputs/analytical_solution/tri_k3_t1e-3/data_rates.dat};
        \addplot [thick, mark=o, mark options=solid, olive, densely dotted] table[x=MeshSize,y=Error] {outputs/analytical_solution/tri_k3_t1e-5/data_rates.dat};
      \end{loglogaxis}            
    \end{tikzpicture}
    \subcaption{Triangular mesh}
  \end{minipage}
  \vspace{0.25cm}\\
  \begin{minipage}{0.45\textwidth}
    \begin{tikzpicture}[scale=0.85]
      \begin{loglogaxis}[legend columns=3, legend to name=conv.hexa.analytical]  
        \addplot [thick, mark=star, red] table[x=MeshSize,y=Error] {outputs/analytical_solution/locref_k0_t1e-1/data_rates.dat};
        \addlegendentry{$k=0$, $t=10^{-1}$}
        \logLogSlopeTriangle{0.90}{0.4}{0.1}{1}{black};
        \addplot [thick, mark=star, mark options=solid, red, dashed] table[x=MeshSize,y=Error] {outputs/analytical_solution/locref_k0_t1e-3/data_rates.dat};
        \addlegendentry{$k=0$, $t=10^{-3}$}
        \addplot [thick, mark=star, mark options=solid, red, densely dotted] table[x=MeshSize,y=Error] {outputs/analytical_solution/locref_k0_t1e-5/data_rates.dat};
        \addlegendentry{$k=0$, $t=10^{-5}$}
        \addplot [thick, mark=*, blue] table[x=MeshSize,y=Error] {outputs/analytical_solution/locref_k1_t1e-1/data_rates.dat};
        \addlegendentry{$k=1$, $t=10^{-1}$}
        \logLogSlopeTriangle{0.90}{0.4}{0.1}{2}{black};
        \addplot [thick, mark=*, mark options=solid, blue, dashed] table[x=MeshSize,y=Error] {outputs/analytical_solution/locref_k1_t1e-3/data_rates.dat};
        \addlegendentry{$k=1$, $t=10^{-3}$}
        \addplot [thick, mark=*, mark options=solid, blue, densely dotted] table[x=MeshSize,y=Error] {outputs/analytical_solution/locref_k1_t1e-5/data_rates.dat};
        \addlegendentry{$k=1$, $t=10^{-5}$}
        \addplot [thick, mark=square*, black] table[x=MeshSize,y=Error] {outputs/analytical_solution/locref_k2_t1e-1/data_rates.dat};
        \addlegendentry{$k=2$, $t=10^{-1}$}
        \logLogSlopeTriangle{0.90}{0.4}{0.1}{3}{black};
        \addplot [thick, mark=square*, mark options=solid, black, dashed] table[x=MeshSize,y=Error] {outputs/analytical_solution/locref_k2_t1e-3/data_rates.dat};
        \addlegendentry{$k=2$, $t=10^{-3}$}
        \addplot [thick, mark=square*, mark options=solid, black, densely dotted] table[x=MeshSize,y=Error] {outputs/analytical_solution/locref_k2_t1e-5/data_rates.dat};
        \addlegendentry{$k=2$, $t=10^{-5}$}
        \addplot [thick, mark=o, olive] table[x=MeshSize,y=Error] {outputs/analytical_solution/locref_k3_t1e-1/data_rates.dat};
        \addlegendentry{$k=3$, $t=10^{-1}$}
        \logLogSlopeTriangle{0.90}{0.4}{0.1}{4}{black};
        \addplot [thick, mark=o, mark options=solid, olive, dashed] table[x=MeshSize,y=Error] {outputs/analytical_solution/locref_k3_t1e-3/data_rates.dat};
        \addlegendentry{$k=3$, $t=10^{-3}$}
        \addplot [thick, mark=o, mark options=solid, olive, densely dotted] table[x=MeshSize,y=Error] {outputs/analytical_solution/locref_k3_t1e-5/data_rates.dat};
        \addlegendentry{$k=3$, $t=10^{-5}$}
      \end{loglogaxis}            
    \end{tikzpicture}
    \subcaption{Loc. Refined mesh}
  \end{minipage}
  \caption{Error $E_h$ (see \eqref{eq:def.error}) with respect to the meshsize $h$ for the analytical solution of Section \ref{sec:test.analytical}.\label{fig:conv.analytical}}
\end{figure}
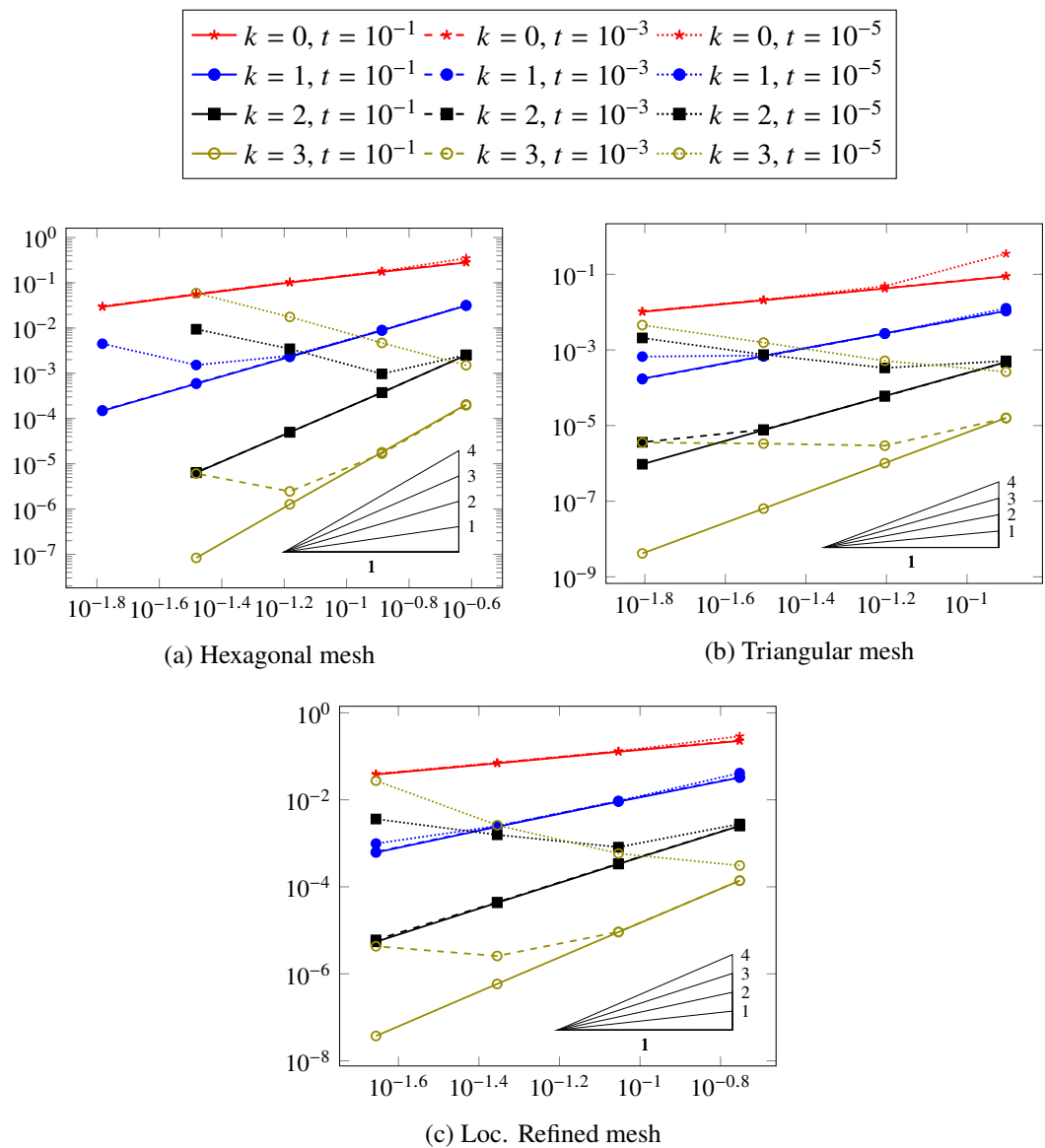

\subsection{Convergence to the Kirchhoff--Love solution in the limit $t\to 0$}\label{sec:tests.kirchoff}

As $t\to 0$, the solution to the Reissner--Mindlin model converges to the solution of the Kirchoff--Love model (see \cite[Section VII.3.1]{Brezzi.Fortin:91} for clamped boundary conditions, and especially Remark 3.4 therein). The tests presented in this section are designed to assess the ability of the lowest-order version $k=0$ of the DDR scheme \eqref{eq:discrete} to capture this behaviour, which is another way of assessing its robustness for small values of $t$. We fix the load $f$ such that the solution to the Kirchoff--Love model is $w(x,y)=\sin(\pi x)\sin(\pi y)$, and we measure the error by using \eqref{eq:def.error} with $(u,\bvec{\theta})=(w,\GRAD w)$. To demonstrate the flexibility of the DDR scheme, the simulations are run using (``soft'') simply supported boundary conditions, which consist in prescribing the displacement $u$ and the normal stress $(\tens{C}\GRADs\bvec{\theta})\normal$ on the boundary of the domain.

The first series of results, using the hexagonal and triangular meshes in Figure \ref{fig:meshes}, are presented in Figure \ref{fig:conv.kl}.
We see in Figure \ref{fig:conv.kl.h} that the scheme is fully robust with respect to $t$, showing a clear $\mathcal O(h)$ convergence rate for $t\in\{10^{-3},10^{-5}\}$;
for $t=10^{-1}$, the model error dominates, explaining the saturation of convergence (the beginning of a model error saturation is also visible for the last triangular mesh and $t=10^{-3}$). In Figure \ref{fig:conv.kl.t} we assess the said model error and see that, for the finest mesh of each family, an $\mathcal O(t)$ convergence rate is achieved until the numerical error dominates.

\begin{figure}\centering
  \begin{minipage}[b]{0.45\textwidth}\centering
    \begin{tikzpicture}[scale=0.85]
      \begin{loglogaxis}[
          legend columns=2,
          legend style={at={(0.5,1.1)},anchor=south}
        ]
        \addplot [thick, mark=star, red] table[x=MeshSize,y=EnError] {outputs/kirchoff_limit/hexa_k0_t1e-1/data_rates.dat};
        \logLogSlopeTriangle{0.90}{0.4}{0.1}{1}{black};
        \addlegendentry{Hex., $t=10^{-1}$}
        \addplot [thick, mark=*, blue] table[x=MeshSize,y=EnError] {outputs/kirchoff_limit/tri_k0_t1e-1/data_rates.dat};
        \addlegendentry{Tri., $t=10^{-1}$}
        \addplot [thick, mark=star, mark options=solid, red, dashed] table[x=MeshSize,y=EnError] {outputs/kirchoff_limit/hexa_k0_t1e-3/data_rates.dat};
        \addlegendentry{Hex., $t=10^{-3}$}
        \addplot [thick, mark=*, mark options=solid, blue, dashed] table[x=MeshSize,y=EnError] {outputs/kirchoff_limit/tri_k0_t1e-3/data_rates.dat};
        \addlegendentry{Tri., $t=10^{-3}$}
        \addplot [thick, mark=star, mark options=solid, red, densely dotted] table[x=MeshSize,y=EnError] {outputs/kirchoff_limit/hexa_k0_t1e-5/data_rates.dat};
        \addlegendentry{Hex., $t=10^{-5}$}
        \addplot [thick, mark=*, mark options=solid, blue, densely dotted] table[x=MeshSize,y=EnError] {outputs/kirchoff_limit/tri_k0_t1e-5/data_rates.dat};
        \addlegendentry{Tri., $t=10^{-5}$}
      \end{loglogaxis}            
    \end{tikzpicture}
    \subcaption{Energy error vs.~$h$ for various $t$\label{fig:conv.kl.h}}
  \end{minipage}
  \begin{minipage}[b]{0.45\textwidth}\centering
    \begin{tikzpicture}[scale=0.85]
      \begin{loglogaxis}[%
          legend columns=2,
          legend style={at={(0.5,1.1)},anchor=south}
        ]
        \addplot [thick, mark=star, red] table[x=t,y=EnError] {outputs/kirchoff_limit/hexa-mesh6-by-t.dat};
        \addlegendentry{Hexagonal}
        \logLogSlopeTriangle{0.90}{0.3}{0.1}{1}{black};
        \addplot [thick, mark=star, mark options=solid, blue, dashed] table[x=t,y=EnError] {outputs/kirchoff_limit/tri-mesh6-by-t.dat};
        \addlegendentry{Triangular}
      \end{loglogaxis}            
    \end{tikzpicture}
    \subcaption{Energy error vs.~$t$ on the finest meshes\label{fig:conv.kl.t}}
  \end{minipage}
  \caption{Energy errors $E_h$ (see \eqref{eq:def.error}) between the solution to the DDR scheme \eqref{eq:discrete} and the solution to the limiting Kirchoff--Love model, on the hexagonal and triangular meshes of Figure \ref{fig:meshes} (on the right, the finest hexagonal mesh has 103,041 elements, and the finest triangular mesh has 229,376 elements).\label{fig:conv.kl}}
\end{figure}
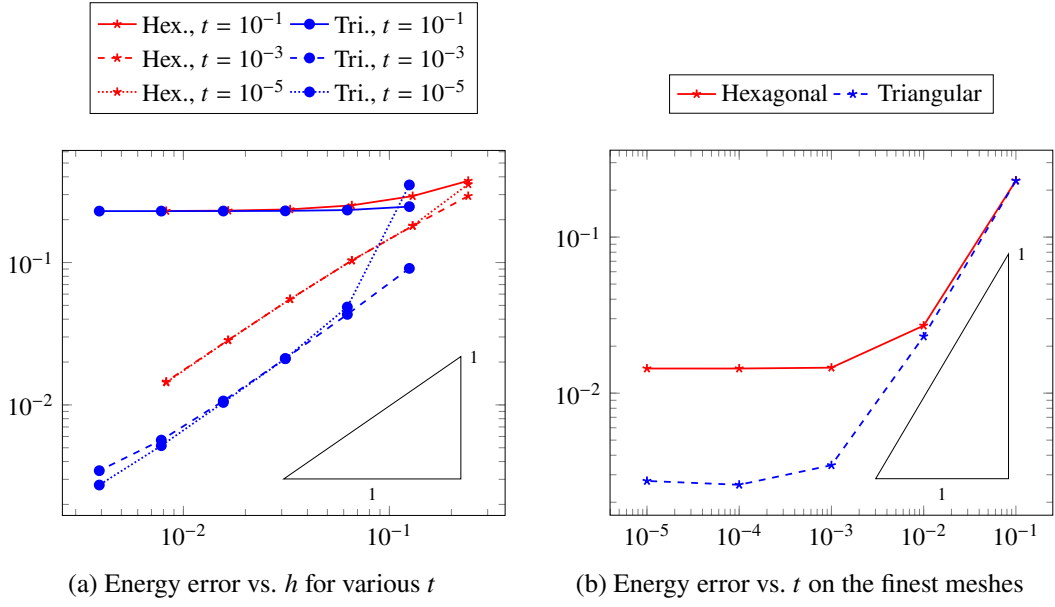

We also numerically compare the DDR scheme \eqref{eq:discrete} (still with $k=0$) to the stabilised $\Poly{2}$-$(\Poly{1}+\mathcal B^3)$ scheme of \cite{Chapelle.Stenberg:98} (where $\mathcal B^3$ is the bubble function). The errors for each of these schemes are measured in slightly different norms, which is natural given their different degrees of freedom; all these norms, however, mimic $H^1$-norms for both the displacement and the rotation. It should also be mentioned that \cite{Chapelle.Stenberg:98} uses a ``hard'' version of the simply supported boundary conditions, which are known to produce less severe boundary layers that the ``soft'' version used in our implementation \cite{Arnold.Falk:96}. In spite of this, the results in Table \ref{tab:comparison.stenberg} (corresponding to a family of uniformly refined triangular meshes, see Figure \ref{fig:unif.tri}) show that both methods produce similar relative errors on the displacement and rotation. This is moreover achieved with fewer degrees of freedom per triangle (12 for the $\Poly{2}$-$(\Poly{1}+\mathcal B^3)$ method, after static condensation of the bubble, vs.~9 for the DDR scheme), showing the potential interest of polygonal methods also on standard meshes.

\begin{figure}\centering
  \begin{minipage}{0.45\textwidth}\centering
    \includegraphics[height=4cm]{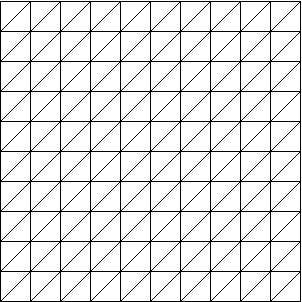}
    \subcaption{Triangular mesh \label{fig:unif.tri}}
  \end{minipage}
  \hspace{0.25cm}
  \begin{minipage}{0.45\textwidth}\centering
    \includegraphics[height=4cm]{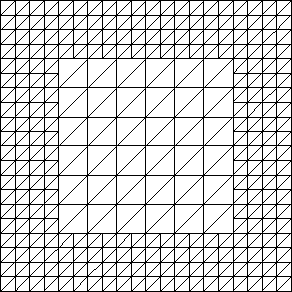}
    \subcaption{Refined triangular mesh \label{fig:refined.tri}}
  \end{minipage}
  \caption{Members of mesh families used in the numerical tests of Section \ref{sec:tests.kirchoff}.\label{fig:meshes.kirchoff}}
\end{figure}

\begin{table}
  \renewcommand*{\arraystretch}{1.2}
  \centering
  \begin{tabular}{c|cc|cc}
    \toprule
     & \multicolumn{2}{c|}{Stabilised $\Poly{2}$-$(\Poly{1}+\mathcal B^3)$ scheme} & \multicolumn{2}{c}{DDR scheme \eqref{eq:discrete}}\\
    \midrule
     nb.~triangles & Error on $u$ & Error on $\bvec{\theta}$ & Error on $u$ & Error on $\bvec{\theta}$\\
     800 & 2.82e-3 & 0.135 & 1.63e-3 & 4.34e-2\\
     3200 & 7.06e-4 & 6.75e-2 & 3.96e-4 & 2.14e-2 \\
     12800 & 1.77e-4 & 3.38e-2 & 1e-4 & 1.07e-2\\
    \bottomrule
  \end{tabular}
  \caption{Comparison, on the meshes \ref{fig:unif.tri}, between the stabilised $\Poly{2}$-$(\Poly{1}+\mathcal B^3)$ scheme of \cite{Chapelle.Stenberg:98} and the DDR scheme \eqref{eq:discrete} for $k=0$: $H^1$-like errors between the approximation solutions (with $t=10^{-3}$) and the solution to the limiting Kirchoff--Love problem.  
  \label{tab:comparison.stenberg}}
\end{table}

One of the main advantages of polytopal methods over usual finite elements lies in the seamless support of nonconforming local mesh refinement, without degradation of the mesh regularity factor and without any special treatment of hanging nodes. For the Reissner--Mindlin model, we know that the rotation variable develops a boundary layer; it therefore makes sense to use meshes that are locally refined around the boundary to better capture this layer without increasing too much the number of degrees of freedom (DOFs). In Table \ref{tab:comparison.stenberg.refined} we compare the stabilised $\Poly{2}$-$(\Poly{1}+\mathcal B^3)$ scheme on the regular triangular meshes \ref{fig:unif.tri} with the DDR scheme on locally refined meshes (see \ref{fig:refined.tri}; these meshes are the \texttt{tri2\_refined\_boundary} meshes in the \texttt{HArDCore2D} repository). As can be seen in this table, the usage of locally refined meshes enables the DDR scheme to capture the boundary layer on $\bvec{\theta}$ as well as the stabilised finite element scheme, but using a vastly reduced number of DOFs. The displacement $u$ is not as well approximated, but this is expected since the meshes are not refined inside the domain, where the displacement gradient is the largest; nevertheless, the bulk of the global energy error (on both variables) is born by $\bvec{\theta}$, and these results demonstrate that the flexibility of the DDR method enables for very efficient approximations of the whole model.

\begin{table}
  \renewcommand*{\arraystretch}{1.2}
  \centering
  \begin{tabular}{ccc|ccc}
    \toprule
     \multicolumn{3}{c|}{Stabilised $\Poly{2}$-$(\Poly{1}+\mathcal B^3)$ scheme} & \multicolumn{3}{c}{DDR scheme \eqref{eq:discrete}}\\
    \midrule
     nb.~DOFs & Error on $u$ & Error on $\bvec{\theta}$ & nb.~DOFs & Error on $u$ & Error on $\bvec{\theta}$\\
      2403 & 2.82e-3 & 0.135 & 550 & 3.44e-2 & 0.127\\
      9603 & 7.06e-4 & 6.75e-2 & 2121 & 8.32e-3 & 5.94e-2\\
      38402 & 1.77e-4 & 3.38e-2 & 8329 & 2.05e-3 & 2.89e-2\\
    \bottomrule
  \end{tabular}
  \caption{Comparison between the stabilised $\Poly{2}$-$(\Poly{1}+\mathcal B^3)$ scheme of \cite{Chapelle.Stenberg:98} (on the meshes \ref{fig:unif.tri}), and the DDR scheme \eqref{eq:discrete} for $k=0$ (on the meshes \ref{fig:refined.tri}): $H^1$-like errors between the approximation solutions (with $t=10^{-3}$) and the solution to the limiting Kirchoff--Love problem.  
  \label{tab:comparison.stenberg.refined}}
\end{table}

\section{Conclusions}

  In this work we have introduced a novel numerical scheme for the Reissner--Mindlin problem with appealing features, the most prominent of which is the support of general polygonal meshes and arbitrary approximation orders.
  We have proved optimal error estimates that are additionally fully robust with respect to the plate thickness for the lowest-order version of the method.
  As far as this lowest-order version is concerned, the numerical tests confirm full robustness and compare well with classical methods when meshes supported by the latter are used;
  through nonconforming mesh refinement, on the other hand, the proposed method can achieve a dramatic reduction of the number of degrees of freedom required to achieve a given accuracy.
  Higher-order versions, although suffering from a degradation of convergence for small values of the thickness, appear interesting in terms of precision versus meshsize when moderate values are considered.
  The loss of precision for small values of the thickness is probably not specific to the method presented in this work, but rather points out a general limitation of higher-order schemes for the specific problem at hand.

\section*{Acknowledgements}

The authors acknowledge the partial support of \emph{Agence Nationale de la Recherche} grant NEMESIS (ANR-20-MRS2-0004).


\printbibliography

\end{document}